\documentclass{amsart}
\usepackage{amssymb,epsfig,amsxtra, amsmath}
\usepackage[mathscr]{eucal}
\usepackage{a4}
\input{xy}
\xyoption{all}
\newtheorem{theorem}{\textbf{Theorem}}[section]
\newtheorem{proposition}[theorem]{\textbf{Proposition}}

\newtheorem{definition}[theorem]{\textbf{Definition}}
\newtheorem{lemma}[theorem]{\textbf{Lemma}}
\newtheorem{corollary}[theorem]{\textbf{Corollary}}
\newtheorem{remark}[theorem]{\textbf{Remark}}
\newtheorem{example}[theorem]{\textbf{Example}}
\def\pencil{base point free pencil trick}

\def\ag{\`a}

\def\p{\mathbb P^2}
\def\n{\mathbb P^{N}}

\def\v{V_{n,g}}
\def\s{\Sigma^n_{k,d}}
\def\d{\delta}

\def\cp{C^{\prime}}
\def\g{\Gamma}

\def\fn{\mathcal N_\phi}
\def\k{\mathcal K_{\phi}}
\def\bnm{\mu_{o,C}}
\def\o{\mathcal O}
\def\m{\mathcal M}
\def\pa{\binom{n-1}{2}}
\def\ge{[\Gamma]}

\def\cp{C^{\prime}}
\def\g{\Gamma}

\def\fn{\mathcal N_\phi}

\def\k{\mathcal K_{\phi}}

\def\o{\mathcal O}

\begin{document}
\title[Number of moduli of irreducible families...]
{Number of moduli of irreducible families of plane curves with
 nodes and cusps.}

\author{ Concettina Galati }

\address{Dipartimento di Matematica, Universit\ag\, degli Studi
della Calabria, via P. Bucci, cubo 30B, Arcavacata di Rende (CS)}

\email{galati@mat.unical.it}

\thanks{}



\date{06 September 2005}

\dedicatory{}

\commby{}


\begin{abstract}
Let $\s\subset\mathbb P^{\frac{n(n+3)}{2}}$ be the family of
irreducible plane curves
of degree $n$ with $d$ nodes and $k$ cusps as singularities. Let $\Sigma\subset\s$
be an irreducible component. We consider the natural rational map
$$
\Pi_{\Sigma}:\Sigma\dashrightarrow \mathcal M_g,
$$
from $\Sigma$ to the moduli space of curves of genus $g=\pa-d-k$.
We define the \textit{number of moduli of $\Sigma$} as the
dimension $dim(\Pi_{\Sigma}(\Sigma))$.
If $\Sigma$ has the expected dimension equal to $3n+g-1-k$, then
\begin{equation}\label{dis}
dim(\Pi_{\Sigma}(\Sigma))\leq min(dim(\m_g), dim(\m_g)+\rho-k),
\end{equation}
where $\rho:=\rho(2,g,n)=3n-2g-6$ is the Brill-Neother number of
the linear series of degree $n$ and dimension $2$ on a smooth
curve of genus $g$. We say that $\Sigma$ has the expected number
of moduli if the equality holds in \eqref{dis}. In this paper
we construct examples
of families of irreducible plane curves with nodes and cusps as singularities
having expected number of moduli and with non-positive Brill-Noether number.  
\end{abstract}

\maketitle
\section[On the number of moduli...]{Introduction}\label{introduction}
In this paper we compute the number of moduli of certain 
families of irreducible plane curves with nodes and 
cusps as singularities. 
Let $\s\subset\mathbb P (H^0(\p,\mathcal{O}_{\p}(n))):=\n$, with
$N=\frac{n(n+3)}{2}$,  be the closure, in the Zariski's topology, of
the locally closed set of reduced and irreducible plane curves of
degree $n$ with $k$ cusps and $d$ nodes. Let $\Sigma\subset\s$ 
be an irreducible component
of the variety $\s$. We denote by $\Sigma_0$ the open set of
$\Sigma$ of points $\ge\in\Sigma$ such that $\Sigma$ is smooth at
$\ge$ and such that $\ge$ corresponds to a reduced and irreducible
plane curve of degree $n$ with $d$ nodes, $k$ cusps and no further
singularities. Since the tautological family $ \mathcal
S_0\to\Sigma_0$, parametrized by $\Sigma_0$, is an equigeneric
family of curves, by normalizing the total space, we get a family
\begin{displaymath}
\xymatrix{\mathcal S^\prime_0 \ar[dr]\ar[r] &
\mathcal S_0\ar[d]\ar@{^{(}->}[r]&\p\times\Sigma_0\ar[dl]\\
&\Sigma_0&}
\end{displaymath}
of smooth curves of genus $g=\binom{n-1}{2}-k-d$. Because of
the functorial properties of the moduli space $\m_g$ of smooth
curves of genus $g$, we get a regular map $\Sigma_0\to\m_g$,
sending every point $[\g]\in\Sigma_0$ to the isomorphism class of
the normalization of the plane curve $\g$ corresponding to the
point $[\g]$. This map extends to a rational map
$$
\Pi_{\Sigma}:\Sigma\dashrightarrow\m_g.
$$

We say that $\Pi_{\Sigma}$ is the \textit{moduli map of} $\Sigma$ and we set
$$
\textit{number of moduli of $\Sigma$}:= \dim (\Pi_{\Sigma}(\Sigma)).
$$
Notice that, when $\s$ is reducible, two different irreducible
components of $\s$ can have different number of moduli. We say
that $\Sigma$ has \textit{general moduli} if $\Pi_{\Sigma}$ is
dominant. Otherwise, we say that $\Sigma$ has \textit{special
moduli}. 
\begin{definition}\label{ubnm}
When $\Sigma$ has the expected dimension equal to $3n+g-1-k$ and $g\geq 2$, we say that $\Sigma$ has the
expected number of moduli if
$$
\dim (\Pi_{\Sigma}(\Sigma))= \min(\dim(\m_g), \dim(\m_g)+\rho-k),
$$
where $\rho:=\rho(2,g,n)=3n-2g-6$ is the number of Brill-Noether
of the linear series of degree $n$ and dimension $2$ on a smooth
curve of genus $g$.
\end{definition}
As we shall see in the next section, when $g\geq 2$ and when $\Sigma$ 
has the expected dimension equal to $3n+g-1-k$,  
the number of moduli of $\Sigma$ is at most equal to the expected one.
This happens in particular if $k<3n$. If $k\geq 3n$, in general we have not 
an upper-bound for the dimension of $\Sigma$ 
and we cannot provide an upper bound for the number of
moduli of $\Sigma$, (see lemma \ref{lemmaubnm} and remark \ref{remarkubnm}). 
Moreover, by classical
Brill-Neother theory when $\rho$ is positive and by a well know
result of Sernesi when $\rho\leq 0$ (see \cite{ser}), we have that
$\Sigma^n_{0,d}$, (which is irreducible by \cite{h1}), has the expected number
of moduli for every $d\leq \pa$. When $k>0$ there are known results giving 
sufficient conditions for the existence of irreducible components $\Sigma$
of $\s$ with general moduli, (see propositions \ref{n>2g-1+2k} and \ref{5.43}
and corollary \ref{5.431}). In this article we construct examples of 
families of irreducible plane  curves with nodes and cusps with finite 
and expected number of moduli.
A large part of this paper is obtained 
working out the main ideas and
techniques that Sernesi uses in \cite{ser}. 

In section \ref{S-E} we introduce the varieties $\s$ and we recall their
main properties. In section \ref{krnm} we discuss on definition \ref{ubnm}
and we summarize known results on the number of moduli
of families of irreducible plane curves with nodes and cusps. In theorem \ref{gknnc}
we prove the existence of plane curves with nodes and cusps as
singularities whose singular points are in sufficiently general
position to impose independent linear conditions to a linear
system of plane curves of a certain degree. This result is related to 
the moduli problem by lemma \ref{gknpc}, remark \ref{kimplies}
and proposition \ref{bnm}, where we find sufficient conditions in order that 
an irreducible component $\Sigma\subset\s$ has the expected number of moduli.
If $\Sigma$ verifies the hypotheses of proposition  \ref{bnm}, 
then the Brill-Neother number $\rho$ is not positive and $\Sigma$ has finite number of moduli. Moreover, by lemma \ref{rank}
and corollary \ref{semicon}, for every $k^\prime\leq k$ and
$d^\prime\leq d+k-k^\prime$, there is at least an irreducible
component $\Sigma^\prime\subset\Sigma_{k^\prime,d^\prime}^n$, such
that $\Sigma\subset\Sigma^\prime$ and the general element
$[D]\in\Sigma^\prime$ corresponds to a  plane curve $D$ verifying 
hypotheses of proposition \ref{bnm} and so having the expected 
number of moduli. Finally, the main result of
this paper is contained in theorem \ref{moduli}, where, by using
induction on the degree $n$ and on the genus $g$ of the general
curve of the family, we construct examples of families of irreducible plane curves with 
nodes and cusps verifying the hypotheses of proposition \ref{bnm}.
In particular, we prove that, if $k\leq 6$
and $\rho\leq 0$, then $\s$ has at least an irreducible component which is not
empty and which has the expected number of moduli. 
This result may be improved and examples of families of curves showing
that the condition $k\leq 6$ is not sharp are given in remark \ref{remarkmoduli}. 
Notice that the previous theorem provides only examples of families 
of plane curves with nodes and cusps with expected number of moduli,
when $\rho$ is not positive. When the number of cusps $k$ is very small,
we expect it is possible to prove the existence of irreducible components of
$\s$ with expected number of moduli, for every value of $\rho$.
For example, from a result of Eisenbud and Harris, it follows 
that $\Sigma_{1,d}^n$, (which is irreducible by \cite{k2}), 
has general moduli if $\rho\geq 2$, (see corollary \ref{5.431}). 
In theorem \ref{rho=1}, by using induction on $n$ we find that  
$\Sigma_{1,d}^n$ has general moduli also when $\rho =1$. By recalling that, 
by theorem \ref{moduli}, $\Sigma_{1,d}^n$ has expected number of moduli 
when $\rho\leq 0$, we conclude that
$\Sigma_{1,d}^n$ has the expected number of moduli for every $\rho$ or, equivalently,
for every $d\leq\pa-1$.
We still don't know
examples of irreducible components of $\s$ 
having number of moduli smaller that the expected.

\section[On the number of moduli...]{Preliminaries}\label{preliminaries}
\subsection{On Severi-Enriques varieties}\label{S-E}We shall denote by $\n=
\mathbb P^{\frac{n(n+3)}{2}}$ the Hilbert scheme of plane curves of degree $n$, by $\ge\in\n$ the point parametrizing a plane curve $\g\subset\p$ and by 
$\s\subset\n$ the 
closure, in the Zariski topology, of the locally closed set parametrizing
reduced and irreducible plane curves of degree $n$ with $d$ nodes and $k$ cusps as singularities. These varieties have been introduced 
at the beginning of the last century by Severi and Enriques. In particular, the case $k=0$
has been studied first by Severi and for this reason
the varieties $\Sigma^{n}_{0,d}$ are usually called Severi varieties,
while for $k>0$ the varieties $\Sigma^n_{k,d}$ are called Severi-Enriques varieties.
We recall that every irreducible component $\Sigma$ of $\Sigma^{n}_{k,d}$ has dimension at least equal to 
$$N-d-2k=3n+g-1-k,$$
where $g=\pa-k-d$. 
When the equality holds \textit{we say that $\Sigma$ has expected dimension.} 
Moreover, it is well known that \textit{if $k<3n$ then
every irreducible component $\Sigma$ of $\Sigma^{n}_{k,d}$ has expected dimension,}
(see for example \cite{w2} or \cite{z2}).
On the contrary, when $k\geq 3n$, there exist examples of irreducible components of $\s$
having dimension greater than the expected, (see \cite{z2}).  
Moreover, we recall that $\Sigma^{n}_{0,d}$ is not empty 
for every $d\leq \pa$ and it contains in its closure all points
parameterizing irreducible plane curves of degree $n$ and genus $g=\pa-d$,
(see \cite{z1}, \cite{z2} and \cite{ac1}). Often, we shall denote $\Sigma^{n}_{0,d}$
by $\v$. While the proof of the existence of $\v$ 
is quite elementary and it is due to Severi, the irreducibility 
of $\v$ remained an open problem for a long time and it has been proved by
Harris only in 1986. Later, by using the same techniques of Harris, Kang has 
proved the irreducibility of $\s$ with $k\leq 3$, see \cite{h1} and \cite{k2}.
However, in general, $\s$ is reducible and there exist values of $n$, $d$ and $k$
such that $\s$ is empty, (see \cite{z2}, \cite{sh2}, \cite{sh1}, \cite{gk} 
or chapter 2 of \cite{tesi} and related references). Finally, we recall
that, if $\Sigma\subset\s$ is a non-empty irreducible component of the expected 
dimension equal to $3n+g-1-k$, then, for every $k^\prime\leq k$ and $d^\prime\leq d+k-k^\prime$, there exists a non-empty irreducible component $\Sigma^\prime\subset\Sigma^n_{k^\prime,d^\prime}$ such that $\Sigma\subset\Sigma^\prime$.
This happens in particular if $k<3n$. 
More precisely, it is true that, if $\g\subset\p$ is a reduced
(possibly reducible) plane curve of degree $n$ with $k<3n$ cusps at points
$q_1,\dots,q_k$, nodes at points
$p_1,\dots,p_d$ and no further singularities, then, chosen arbitrarily $k_1$ cusps, say
$q_1,\dots,q_{k_1}$ among the $k$ cusps of $\g$, $k_2$ cusps 
$q_{k_1+1},\dots,q_{k_2}$ among
$q_{k_1+1},\dots,q_k$ and $d_1$ nodes $p_1,\dots,p_{d_1}$
among the nodes of
$\g$, there exists a family of reduced plane curves $\mathcal D\to B\subset \n$
of degree $n$, whose special fibre is $\mathcal D_0=\g$ and whose general fibre 
$\mathcal D_t=D$ has a node 
in a neighborhood of every marked node of $\g$, a cusp in a neighborhood of
each point $q_1,\dots,q_{k_1}$,
a node in a neighborhood of each point $q_{k_1+1},\dots,q_{k_2}$ and no further 
singularities, (see \cite{z2}, corollary 6.3 of \cite{gk} 
or lemma 3.17 of chapter 2 of \cite{tesi}). To save space,
we shall say that \textit{the family $\mathcal D\to B$ is obtained from
$\g$ by preserving the singularities $q_1,\dots,q_{k_1}$ and $p_1,\dots,p_{d_1}$,
by deforming in a node each cusp $q_{k_1+1},\dots,q_{k_2}$ and by smoothing the other
singularities.}   
\subsection{Known results on the number of moduli of $\s$}\label{krnm}
In order to explain the definition \ref{ubnm}, we need to recall
some basics of Brill-Noether theory. Given a smooth curve $C$ of
genus $g$, the set $G^2_n(C)$ of linear series $g^2_n$ on $C$
of dimension $2$ and degree $n$, is a projective variety which
verifies the following properties:
\begin{enumerate}
	\item $G^2_n(C)$  is
  not empty of dimension at least $\rho$, if
  $\rho (2, n, g) = 3n-2g-6\geq 0$, (see theorem V.1.1 and proposition
  IV.4.1 of \cite{acgh}).
  \item Let $g^2_n$ be a given
linear series, let $H\in g^2_n$ be a divisor and let $W\subset
H^0(C,H)$ be the three dimensional vector space corresponding to $g^2_n$.
Denoting by $\omega_C=\o_C(K_C)$ the canonical sheaf of $C$ and by
$$
\mu_{o,C}:W\otimes H^0(C,\omega_C(-H))\to H^0(C,\omega_C)
$$
the natural multiplication map, also called the \textit{Brill-Noether map} 
of the pair $(C,W)$, we have that the dimension of the tangent space to
$G^2_n(C)$ at the point $[g^2_n]$,corresponding to $g^2_n$, is
equal to
$$
\dim (T_{[g^2_n]}G^2_n(C))=\rho+\dim (\ker(\mu_{0,C})),
$$
(see \cite{ac2} or proposition IV.4.1 of \cite{acgh} for a proof). 
\item Moreover, if $C$ is a
curve with general  moduli  (i.e. if $[C]$ varies in an open set of
$\m_g$), the variety $G^2_n(C)$ is empty if $\rho<0$, it consists
of a finite number of points if $\rho=0$ and it is reduced,
irreducible, smooth and not empty variety of dimension exactly
$\rho$, when $\rho\geq 1$, (see theorem V.1.5 and theorem V.1.6 of \cite{acgh}).
In the latter case, the general $g^2_n$ on $C$
defines a local embedding on $C$ and it maps $C$ to $\p$ as a
nodal curve, (see theorem 3.1 of \cite{ac1} or lemma 3.43 of \cite{h2}).\label{g2n2}
\end{enumerate}
From \eqref{g2n2}, we deduce that, the Severi variety
$\Sigma^n_{0,d}=V_{n,g}$ of irreducible plane curves of genus
$g=\pa-d$, has general moduli when $\rho\geq
0$ and it has special moduli when $\rho<0$. When $\rho < 0$, and
then $g\geq 3$, by definition \ref{ubnm}, we expect that the image 
of $\v$ into $\m_g$ has
codimension exactly $-\rho$. Equivalently, recalling that, in this
case, $$\dim (\v)=3n+g-1=3g-3+\rho+8=\dim(\m_g)+\rho+\dim(Aut(\p)),$$
we expect that on the smooth curve $C$, obtained by normalizing the
plane curve corresponding to the general element of $\v$, there
is only a finite number of $g^2_n$ mapping $C$ to the plane as a
nodal curve. This is a well known result proved by Sernesi in
\cite{ser}.
\begin{theorem}[Sernesi, \cite{ser}]
The Severi variety $V_{n,g}=\Sigma^{n}_{0,d}$ of irreducible plane curves of degree
$n$ and genus $g=\pa-d$ has number of moduli equal to $$
\min(\dim(\m_g),\dim(\m_g)+\rho).$$
\end{theorem}
What can we say about the number of moduli of an irreducible
component $\Sigma$ of $\s$, when $k>0$? 
In this case we need to distinguish the two cases $k<3n$ 
and $k\geq 3n$. In the first case we have the following result.
\begin{lemma}\label{lemmaubnm}
For every not empty irreducible component $\Sigma$ of $\s$,
with $k<3n$ and $g=\pa-k-d\geq 2$, the number of moduli of $\Sigma$ 
is at most equal to 
$$
\min (\dim(\m_g),\,\dim(\m_g)+\rho-k),
$$
where $\rho=3n-2g-6$ is the Brill-Neother number of moduli of
linear series of dimension $2$ and degree $n$ on a smooth curve of
genus $g$. 
\end{lemma}
\begin{proof}
We recall that an ordinary cusp $P$ of a plane curve
$\g$ corresponds to a simple ramification point $p$ of the
normalization map $\phi:C\to \g$, i.e. to a simple zero of the
differential map $d\phi$. If we denote by $G^2_{n,k}(C)\subset
G^2_{n}(C)$ the set of $g^2_n$ on $C$ defining a birational
morphism with $k$ simple ramification points, then 
$G^2_{n,k}(C)$ is a locally closed subset of $G^2_n(C)$ and every
irreducible component $G$ of $G^2_{n,k}(C)$ has dimension at least
equal to $\rho-k$, if it is not empty. In particular, if $ F^2_{n,k}(C)$
is the variety whose points correspond to the pairs
$([g^2_n],\{s_0,s_1,s_2\})$ where $[g^2_n]\in G^2_{n,k}(C)$ and
$\{s_0,s_1,s_2\}$ is a frame of the three dimensional space
associated to the linear series $g^2_n$, then every irreducible
component of $F^2_{n,k}(C)$ has dimension at least equal to 
$$\min(8, \rho-k+8).$$
Now, let $\Sigma$ be
one of the irreducible components of $\s$
and let $\ge$ be a general point of $\Sigma$.
Then, if $\g\subset\p$ is the corresponding plane curve and 
$\phi:C\to\g$ is the normalization map, then the fibre over 
the point $[C]\in\m_g$ of the moduli map
$$
\Pi_\Sigma:\Sigma\dashrightarrow\m_g
$$
consists of an open set in one or more irreducible components
of $F^2_{n,k}(C)$. In particular, \textit{every
irreducible component of the general
fibre of} $\Pi_\Sigma$ \textit{has dimension at least equal to} $\min(8,\rho-k+8)$.
Moreover, if $k< 3n$ then $\Sigma$ has the expected dimension 
equal to $N-d-2k=3n+g-1-k$, (see \cite{z2} or \cite{w2}). Finally, if
$g=\pa-k-d\geq 2$, then
$$
\dim(\Sigma)=3n+g-1-k=3g-3+\rho-k+8.
$$
This proves the statement.
\end{proof}
\begin{remark}\label{remarkubnm} The proof of the 
previous lemma still holds if $k\geq 3n$
but $\Sigma$ has the expected dimension. However in general, when $k\geq 3n$, we
don't have a bound for $\dim(\Pi_\Sigma(\Sigma))$. Indeed, in this case
the dimension of the general fibre of the moduli map of $\Sigma$ is still at
least equal to $\rho-k+8$, but $\Sigma$ may have dimension larger
than $3n+g-1-k$. Anyhow, by  the following proposition,
every not empty irreducible component of $\s$ has special
moduli if $k\geq 3n$.
\end{remark}
\begin{proposition} [Arbarello-Cornalba, \cite{ac1}]
Let $C$ be a general curve of genus $g\geq 2$ and let $\phi:C\to\p$ be
a birational morphism, then the degree of the zero divisor of the differential
map of $\phi$ is smaller than $\rho$. In particular, every irreducible
component of $\s$ has special moduli if $\rho=3n-2g-6<k$.
\end{proposition}

A sufficient condition for the existence of irreducible families
of plane curves with nodes and cusps with
general moduli is given by the following result.
\begin{proposition}[Kang, \cite{k1}]\label{n>2g-1+2k}
$\s$ is irreducible, not empty and with general moduli if
$n>2g-1+2k$, where $g=\pa-d-k$.
\end{proposition}
Actually, in \cite{k1}, Kang proves that if $n>2g-1+2k$, then $\s$
is not empty and irreducible. But from his proof it follows that,
under the hypothesis of proposition \ref{n>2g-1+2k}, $\s$ has
general moduli because the general element of $\s$ corresponds to
a curve which is a projection of an arbitrary smooth curve $C$ of
genus $g$ in $\mathbb P^{n-g}$, from a general $(n-3)$-plane
intersecting the tangent variety of $C$ in $k$ different points.
Another result which may be used to find examples of families of
plane curves with nodes and cusps having general moduli is the
following. Let $g^r_n$ be a linear series on $C$ associated to
a $(r+1)$-space $W\subset H^0(C,\mathcal L)$, where $\mathcal L$
is an invertible sheaf on $C$, and let $\{s_0,\dots,s_r\}$ be a
basis of $W$, then the ramification sequence of the $g^r_n$ at $p$ is
the sequence $b=(b_0,\dots,\,b_r)$ with $b_i=ord_p s_i-i.$
Choosing another basis of $W$, the ramification sequence of
$g^r_n$ at $p$ doesn't change. We say that the ramification
sequence of the $g^r_n$ at $p$ is at least equal to
$b=(b_0,\dots,\,b_r)$ if $b_i\leq ord_p s_i-i$, for every $i$, and
we write $(ord_p s_0,\dots,ord_p s_r-r)\geq (b_0,\dots,\,b_r)$. 
\begin{proposition}[Proposition 1.2 of \cite{eh2}]\label{5.43}
Let $C$ be a general curve of genus $g$, let $p$ be a general
point on $C$ and let $b=(b_0,\dots,\,b_r)$ be any ramification
sequence. There exists a $g^r_n$ on $C$ having ramification at
least $b$ at $p$ if and only if
$$
\sum_{i=0}^r(b_i+g-n+r)_+\leq g,
$$
where $(-)_+:=max(-,0)$.
\end{proposition}
From proposition \ref{5.43}, we easily deduce the following result.
\begin{corollary}\label{5.431}
Suppose that $k\leq 3$ and $\rho=3n-2g-6\geq 2k$. Then 
$\Sigma^n_{k,d}$ is not empty, irreducible and it has general moduli.
\end{corollary}
\begin{proof}
By \cite{k2}, the variety $\Sigma_{k,d}^n$ is irreducible for every $k\leq 3$ 
and $d\leq\pa-k$. Moreover, by using classical arguments, one can prove that 
$\s$ is not empty if $k\leq 4$ and $d\leq \pa-4$, (see, for example, 
corollary 3.18 of chapter two of \cite{tesi}).   
Finally, by theorem 1.1 of \cite{sh3}, by using the terminology of proposition
\ref{5.43}, under the hypothesis $k\leq 3n-4$, in particular if $k\leq 3$, the variety $\Sigma_{k,d}^n$ contains every point of $\n$
corresponding to a plane curve $\g$ of genus $g=\pa-k-d$ such that the
normalization morphism of $\g$ has at least a ramification point
with ramification sequence $(b_0,b_1,b_2)\geq (0,k,k)$. Then, by
proposition \ref{5.43}, if $\rho\geq 2k$ and $k\leq 3$, the moduli map of
$\Sigma_{k,d}^n$ is surjective.
\end{proof}

\section[On the number of moduli...]{On the existence of certain
families of plane curves with nodes and cusps in sufficiently
general position}\label{ote} As we already observed, we don't have a complete answer 
for the existence problem of $\s$. In this section we are
interested in a little more specific existence problem. We shall
prove the existence of plane curves with nodes and cusps as
singularities whose singular points are in sufficiently general
position to impose independent linear conditions to a linear
system of plane curves of a certain degree. 
\begin{definition}
A projective curve $C\subset\mathbb P^r$ is said to be
\emph{geometrically $t$-normal} if the linear series cut out on
the normalization curve $\tilde C$ of $C$ by the pull-back to $\tilde C$
of the linear system of hypersurfaces of $\mathbb P^r$ of degree $t$ is complete. 
\end{definition}
From a geometric point of view, a projective curve $C\subset
\mathbb P^r$ is geometrically $t$- normal if and only if the image curve
$\nu_{t,r}(C)$ of $C$ by the Veronese embedding $\nu_{t,r}:\mathbb
P^r\to\mathbb P^{\binom{r+t}{t}}$ of degree $t$, is not a
projection of a non-degenerate curve living in a higher
dimensional projective space. We shall say that a curve is
\textit{geometrically linearly normal} (\textit{g.l.n.} for short)  
if it is geometrically 1-normal.
Every such a curve $C$ is not a projection of a curve lying in a
projective space of larger dimension.

The following result is proved under more general hypotheses in \cite{ag},
theorem 2.1. 
\begin{lemma}\label{gknpc}
Let $\Gamma\subset\p$ be an irreducible and reduced plane curve of
degree $n$ and genus $g$ with at most nodes and cusps as
singularities. Let $t$ be an integer such that $n-3-t<0$, then
$\g$ is geometrically $t$-normal if and only if it is smooth. On
the contrary, if $n-3-t\geq 0$, the plane curve $\g$ is
geometrically $t$-normal if and only if its singular points impose
independent linear conditions to plane curves of degree $n-3-t$.
\end{lemma}
We recall the following classical definition.
\begin{definition}\label{adjointdivisor}
Let $\g\subset\p$ be a plane curve of degree $n$ with $d$ nodes at
$p_1,\dots,p_d$ and $k$ cusps at $q_1,\dots,q_k$ as singularities.
Let $\phi:C\to \g$ be the normalization of $\g$. 
The adjoint divisor $\Delta$ of $\phi$ is the divisor on $C$ defined 
by $\Delta=\sum_{i=1}^{d}\phi^{-1}(p_i)+
\sum_{j=1}^{k}2\phi^{-1}(q_j)$. 
\end{definition}
\begin{proof}[Proof of lemma \ref{gknpc}.]
Let $\g$ be a plane curve as in the statement of the lemma. Then,
$\g$ is geometrically $t$-normal if
and only if, by definition,
$$
h^0(C,\o_C(t))=h^0(\p,\o_{\p}(t))-h^0(\p,\mathcal I_{\g}(t))
$$
where $\mathcal I_{\g}$ is the ideal sheaf of $\g$ in $\p$
and $\o_C(t):=\o_C(t\phi^*(H))$, where $H$ is the general
line of $\p$. By
Riemann-Roch theorem, $\g$ is geometrically $t$-normal if and only
if
\begin{equation}\label{equgtn}
h^0(C,\omega_C(-t)))=-nt+g-1+\frac{(t+1)(t+2)}{2}-h^0(\p,\mathcal
I_{\g}(t)),
\end{equation}
where $g$ is the geometric genus of $\g$ and $\omega_C$ is the
canonical sheaf of $C$. On the other hand, it is well known that 
$H^0(C,\omega_C(-t))=H^0(C,\o_C(n-3-t)(-\Delta))$, where $\Delta$ 
is the adjoint divisor of $\phi$, (see definition \ref{adjointdivisor}
and \cite{acgh}, appendix A). 
If $n-3-t < 0$ then
$h^0(C,\o_C(n-3-t))=0$ and $\g$ is geometrically $t$-normal if and only
if
$$
h^0(\p,\o_{\p}(t))-h^0(\p,\mathcal
I_{\g}(t))=n t-\frac{n^2-3n}{2}+\delta,
$$
where $\delta=\pa -g=\deg(\Delta)/2$. This equality is verified if
and only if $\delta=0$, i.e. $\g$ is smooth. If $n-3\geq t$,
$h^0(\p,\mathcal I_{\g}(t))=0$ and \eqref{equgtn} is verified if
and only if
$$
h^0(C,\o_C(n-3-t)(-\Delta))=h^0(\p,\o_{\p}(n-3-t))-\delta.
$$
On the other hand, if $\psi: S\to\p$ is the blowing-up of
the plane at the singular locus of $\g$, denoting by
$\sum_i E_i$ the pullback of the singular locus of $\g$ with
respect to $\psi$ and by $\o_S(r)$ the sheaf $\o_S(r\psi^*(H))$, 
we have that
$$h^0(C,\o_C(n-3-t)(-\Delta))=
h^0(S,\o_{S}(n-3-t)(-\sum_i E_i))=h^0(\p,\o_{\p}(n-3-t)\otimes A)$$
where $A$ is the ideal sheaf of singular
points of $\g$.
\end{proof}
\begin{remark}\label{kimplies}
Notice that, if an irreducible and reduced plane curve $\g$ of
degree $n$ with only nodes and cusps as singularities is
geometrically $t$-normal, with $t\leq n-3$, then it is geometrically
$r$-normal for every $r\leq t$. Indeed, if a set of points imposes independent linear conditions to a linear system $S$, then it imposes independent linear conditions 
to every linear system $S^\prime$ containing $S$.
\end{remark}
\begin{theorem}\label{gknnc}
Let $\s$ be the variety of irreducible and reduced plane curves of
degree $n$ with $d$ nodes and $k$ cusps. Suppose that $d$, $k$, $n$
and $t$ are such that
\begin{eqnarray}
d + k &\leq & \frac{n^2-(3+2t)n+2+t^2+3t}{2}=
h^0(\mathcal{O}_{\p}(n-t-3))\label{n+c}\\
t & \leq & n-3\,\,\hbox{if}\,\, k=0,\\
\label{c} k  & \leq &  6\,\,\hbox{if}\,\,t= 1,\,2\,\, \hbox{and}\\
\label{c1} k  &\leq &  6+[\frac{n-8}{3}],\,\,\hbox{if}\,\,t=3,
\end{eqnarray}
where $[-]$ is the integer part of $-$. Then the variety $\s$ is
not empty and there exists at least an irreducible component
$W\subset\s$ whose general element corresponds to a geometrically
$t$-normal plane curve.
\end{theorem}
\begin{remark}\label{notsharpness}
As we shall see in the next section, (see proposition \ref{bnm}), the geometric linear normality
of the plane curve corresponding to the general element of an irreducible
component $\Sigma$ of $\s$, is related with the
number of moduli of $\Sigma$. Another motivation for the previous 
theorem has been the family of irreducible plane sextics with six cusps.
By \cite{z2}, we know that $\Sigma^6_{6,0}$
contains at least two irreducible components $\Sigma_1$
and $\Sigma_2$. The general  point of $\Sigma_1$ corresponds to a
sextic with six cusps on a conic, whereas the general element of
$\Sigma_2$ corresponds to a sextic with six cusps not on a conic.
Note that, by the previous lemma the general element of $\Sigma_2$
parameterizes a geometric linearly normal sextic, unlike the
general element of $\Sigma_1$, which corresponds to a projection
of a canonical curve of genus four. Theorem \ref{gknnc}, proves in particular that,
under a suitable restriction, (see inequality \eqref{n+c}),
on the genus of the curve corresponding to the general element 
of the family and, if the number of the cusps is small,
the variety $\s$ contains a not empty irreducible component whose general element
corresponds to a curve which is not a projection of an other curve,
lying in a projective space of larger dimension. 
We notice that the inequality
\eqref{n+c} of the previous theorem can't be improved. Indeed, if
$g=\pa-k-d$, then $k+d> h^0(\p,\o_{\p}(n-3-t))$ if and only if
$g<\frac{2tn-t^2-3t}{2}$. On the other hand, by using the 
same notation as in theorem \eqref{gknnc}, if $g<\frac{2tn-t^2-3t}{2}$,
then, by Riemann-Roch theorem, we have that $h^0(C,\o_C(t))\geq tn-g+1>
\frac{t^2+3t}{2}+1=h^0(\p,\o_{\p}(t)).$ On the contrary, inequalities
\eqref{c} and \eqref{c1} are not sharp, (see example \ref{cc1}). 
\end{remark}

In the case of $k=0$ and $t=1$, theorem \ref{gknnc} has been
proved by Sernesi in \cite{ser}, section 4. The case $k=0$ and
$t\leq n-3$ is already contained in \cite{ag}. To show theorem
\ref{gknnc}, we proceed by induction on the degree $n$ and on the
number of nodes and cusps of the curve. The geometric idea at
the base of the induction on the degree of the curve is, mutatis 
mutandis, the same as that of Sernesi.

\begin{proof}[ Proof of theorem
\ref{gknnc}.] Let $t$ be a positive integer such that $n-3-t\geq 0$ and let
$W\subset\s$ be an irreducible component of $\s$. By standard semicontinuity 
arguments it follows that,
\textit{if there exists a point $[C]\in W$ corresponding to a geometrically
$t$-normal curve with only $k$ cusps and $d$ nodes as
singularities, then the general element of $W$ corresponds to
a geometrically $t$-normal plane curve}.
Moreover, \textit{if the theorem is true for fixed $n$, $t\leq n-3$, $k$
as in \eqref{c} or in \eqref{c1} and $k+d$ as in \eqref{n+c}, then
the theorem is true for $n$, $t$ and any ${k}^\prime\leq k$ and
${d }^\prime\leq d+k-k^\prime$}. Indeed, from the hypotheses \eqref{n+c}, \eqref{c} and
\eqref{c1}, it follows in particular that $k<3n$. By section \ref{S-E}, under this hypothesis, for every $k^\prime\leq k$ and for every 
$d^\prime\leq d+k-k^\prime$, there exists a 
family of plane curves $\mathcal C\to\Delta$ of degree $n$,
parametrized by a 
curve $\Delta\subset\Sigma_{k^\prime, d^\prime}^{n}$, whose special
fibre is $\mathcal C_0=C$ and whose general fibre $\mathcal C_z$ has 
$d^\prime$ nodes and $k^\prime$ cusps as singularities. 
The statement follows by applying the semicontinuity theorem to the family $\tilde{\mathcal C}\to\tilde{\Delta}$, obtained by normalizing the total space of
the pull-back family of $\mathcal C\to\Delta$ to the normalization curve
$\tilde \Delta$ of $\Delta$. \textit{Finally, it's enough
to show the theorem when the equality holds in \eqref{c},
\eqref{c1} and \eqref{n+c}}.

First of all we consider the case $k=0$. We will show the
statement for any fixed $t$ and by induction on $n$. Let, then
$t\geq 1$ and $n=t+3$. In this case the equality holds in
\eqref{n+c} if $d=1=h^0(\p,\mathcal O_{\p}).$ Since one point
imposes independent linear conditions to regular functions, by
using lemma \ref{gknpc}, we find that every irreducible plane
curve of degree $n=t+3$ with one node and no further singularities is
geometrically $t$-normal. So, the first step of the induction is
proved. Suppose, now, that the theorem is true for $n=t+3+a$ and let
$[\Gamma]\in\v$ be a point corresponding to a geometrically
$t$-normal curve with $\frac{a^2+3a+2}{2}$ nodes. Let $D$ be a
line which intersects transversally $\Gamma$ and let $P_1, ...,
P_{t+1}$ be $t+1$ marked points of $\Gamma\cap D$. If
$\Gamma^\prime =\Gamma\cup D\subset\p$, then $P_1, ...,P_{t+1}$
are nodes for $\Gamma^\prime$. Let $C\to \Gamma$ be the
normalization of $\Gamma$ and $\cp \to \Gamma^\prime$ the partial
normalization of $\Gamma^\prime $, obtained by smoothing all
singular points of $\Gamma^\prime$, except $P_1, ...,P_{t+1}$. We
have the following exact sequence of sheaves on $\cp$
\begin{equation}\label{se}
0\to\mathcal O_{D}(t)(-P_1- ...-P_{t+1})\to \mathcal
O_{\cp}(t)\to\mathcal O_{C}(t)\to 0,
\end{equation}
where $O_{\cp}(t):=\o_{\cp}(tH)$ and $H$ is the pull-back with
respect to $C^\prime\to\g^\prime$ of general
line of $\p$. Since
$
\deg(\mathcal O_D(t)(-P_1- ...-P_{t+1}))<0,
$
we get that
$$
h^0(D,\mathcal O_D(t)(-P_1- ...-P_{t+1}))=0
$$
and so
\begin{equation}\label{ea1}
h^0(\cp, \mathcal O_{\cp}(t)) =  h^0(C, \mathcal O_{C}(t))
= h^0(\p,\mathcal O_{\p}(t)).
\end{equation}
Now, by section \ref{S-E}, we can obtain $\Gamma^\prime$ as the limit of a 
$1$-parameter family 
of irreducible plane curves
$$
\psi : \mathcal C \to \Delta\subset\mathbb P^{\frac{(n+1)(n+4)}{2}}
$$
of degree $n+1=t+a+4$ with
$$
\frac{a^2+3a+2}{2}+n-t-1=\frac{(a+1)^2+3(a+1)+2}{2}=h^0(\p,\mathcal
O_{\p}(n+1-t-3))
$$
nodes specializing to nodes of $\Gamma^\prime$ different from the
marked points $P_1, ...,P_{t+1}$. Moreover, one can prove that $\Delta$ is smooth,
(see \cite{z1} or \cite{z2}). Normalizing $\mathcal C$, we obtain a family
whose general fibre is smooth and whose special fibre is exactly $\cp$, and we conclude the
inductive step by \eqref{ea1} and by semicontinuity theorem.

Now we consider the case $t=1,2$ or $3$ and $k$ as in \eqref{c}
and in \eqref{c1}. Suppose the theorem is true for $n$ and let
$\ge\in\s$ be a general point in one of the irreducible components of
$\s$. Then, let $D$ be a smooth plane
curve of degree $t$ if $t=1,\,2$ or an irreducible cubic  with a
cusp if $t=3$. By the generality of $\g$, we may suppose that $D$
intersects $\g$ transversally. Let $P_1, ...,P_{t^2+1}$ be $t^2+1$
fixed points of $\Gamma\cap D$. If $\Gamma^\prime =\Gamma\cup D$,
then $P_1, ...,P_{t^2+1}$ are nodes for $\Gamma^\prime$. Let $C\to
\Gamma$ be the normalization of $\Gamma$ and $\cp \to
\Gamma^\prime$ the partial normalization of $\Gamma^\prime$,
obtained by smoothing all singular points except $P_1,
...,P_{t^2+1}$. By using the same notation and by arguing as before,
from the following exact sequence of sheaves on
$\cp$
$$
0\to\mathcal O_{D}(t)(-P_1- ...-P_{t^2+1})\to \mathcal
O_{\cp}(t)\to\mathcal O_{C}(t)\to 0,
$$
we deduce that
\begin{equation}\label{ea2}
h^0(\cp, \mathcal O_{\cp}(t))= h^0(C, \mathcal O_{C}(t))=
  h^0(\p, \mathcal O_{\p}(t)).
\end{equation}
Now, by section \ref{S-E}, we can obtain $\Gamma^\prime$ as
limit of a family
of irreducible plane curves
$$
\phi : \mathcal C \to \Delta
$$
of degree $n+t$ with $d
+nt-t^2-1=\frac{(n+t)^2+(3+2t)(n+t)+t^2+3t+2}{2}$ nodes
specializing to nodes of $\Gamma^\prime$ different to $P_1,
...,P_{t^2+1}$, and $k+\frac{t^2-3t+2}{2}$ cusps specializing to
cusps of $\Gamma$. We conclude by \eqref{ea2} and by semicontinuity, as before. Now we have
to show the first step of the induction.  For $t=1$ the induction
begins with the cases $(n,k)=(4,1),\,(5,3),\,(6,6)$. Trivially, if
$n=4$ and $k=1$ one point imposes independent conditions to the
linear system of regular functions. If $n=5$ and $k=3$ we have to
show that there are irreducible quintics with three cusps not on a
line. A quintic with three cusps is a projection of the rational normal 
quintic $C_5\subset\mathbb P^5$ from a plane generated by three points lying on three different tangent lines to $C_5$.
By Bezout theorem the three cusps of such a
plane curve can't be aligned. If $n=k=6$, one can repeat the classical
argument used by Zariski, see \cite{z1} or
example 3.20 of chapter 2 of \cite{tesi}. For $t=2$ we
have to show the theorem for
$(n,k)=(5,1),\,(6,3),\,(7,6),\,(8,6)$, while for $t=3$
we have to show the theorem for $(n,k)=(6,1),\,(7,3),\,(8,6),\,(9,6),\,(10,6)$.
The case $t=2$ and $(n,k)=(5,1)$ is trivial.
When $t=2$, $n=6$ and $k=3$ we have that $n-3-t=1$. To show that there exists an
irreducible sextic with three cusps not on a line, consider a
rational quartic $C_4$ with three cusps, (see corollary 3.18
of chapter 2 of \cite{tesi} for the existence). By
Bezout theorem, the three double points of $C_4$ can't be aligned.
Then consider a sextic $C_6$ which is union of $C_4$ and a conic
$C_2$ which intersects $C_4$ transversally. By section \ref{S-E},
one can smooth the intersection
points of $C_4$ and $C_2$ obtaining a family of sextics with three cusps not on a line. 
For $t=2$, $n=7$ and $k=6$ we argue as in the previous case,
by using a sextic $C_6$ with six cusps not on a conic and a
line $R$ with intersects $C_6$ transversally. Similarly for $t=2$ , $n=8$
and $k=6$ and $t=3$ and $(n,k)=(6,1),\,(7,3),\,(8,6),\,(9,6),\,(10,6).$
\end{proof}
\begin{example}\label{cc1}
Inequalities \eqref{c} and \eqref{c1} are not sharp. To see this, we can
consider the example of curves of degree $10$. We recall that we say that
a plane curve is geometrically linearly normal (g.l.n. for short) if it is 
geometrically $1$-normal. Theorem
\ref{gknnc} ensures the existence of g.l.n.
irreducible plane curves of degree $10$ with $k\leq 6$ cusps and
nodes as singularities. But, by using the same ideas
as we used in theorem \ref{gknnc}, one can prove the existence of
g.l.n. plane curves of degree $10$ with nodes and
$k\leq 9$ cusps. It is enough to consider a sextic $\g_6$ with six
cusps not on a conic and a rational quartic $\g_4$ with three cusps
intersecting $\g_6$ transversally.
We choose five points $P_1,\dots,\,P_5$ of $\g_4\cap \g_6$. If
$\g^\prime_6$ and $\g^\prime_4$ are the normalization curves of
$\g_6$ and $\g_4$ respectively and $C^\prime$ is the partial
normalization of $\g_6\cup\g_4$ obtained by normalizing all its
singular points except $P_1,\dots,\,P_5$, by considering the
following exact sequence
$$
0\to\o_{\g^\prime_4}(1)(-P_1-\dots-P_5)\to\o_{C^\prime}(1)\to\o_{\g_6^\prime}(1)\to
0$$ we find that $h^0(C^\prime,\o_{C^\prime}(1))=3$. 
By using terminology of section \ref{S-E}, the statement
follows by smoothing the singular points $P_1,\dots,\,P_5$ of $\g_6\cup\g_4$,
and by semicontinuity, as in the proof of theorem \ref{gknnc}. The bound on the number of cusps of theorem \ref{gknnc} can be
improved also for $t=2$ or $t=3$. For example, theorem \ref{gknnc}
ensures the existence of geometrically $3$-normal curves of degree
$12$ with $k\leq 6$ and nodes as further singularities. But, by
considering a geometrically $3$-normal curve of degree $8$ with
six cusps and a quartic with $3$ cusps and arguing as before, we
can find geometrically $3$-normal irreducible plane curves of degree
$12$ with nodes and $k\leq 9$ cusps.
\end{example}

\section[On the number of moduli...]{Families of plane curves with nodes and
cusps with finite and expected number of moduli.}\label{sectionmoduli} 
Let $\Sigma\subset\s$ be an irreducible component of $\s$.
We want to give sufficient conditions for $\Sigma$ to 
have the expected number of moduli.
Let $[\g]\in\Sigma$ be a general element,
corresponding to a plane curve $\g$ with normalization map
$\phi:C\to\g$.  We shall denote by $\omega_C$ 
the canonical sheaf of $C$ and by $\o_C(1)$ the sheaf associated to the
pullback to $C$ of the divisor cut out on $\g$
from the general line of $\p$. 
\begin{proposition}\label{bnm}
Let $\Sigma\subset\s$ be an irreducible
component of $\s$ and let $[\g]\in\Sigma$ be a general element,
corresponding to a plane curve $\g$ with normalization map
$\phi:C\to\g$. Suppose that $\Sigma$ is smooth of the expected dimension equal 
to $3n+g-1-k$ at $\ge$. Moreover, suppose that:
\begin{enumerate}
\item $\g$ is geometrically linearly normal, i.e. $h^0(C,\o_C(1))=3$,\label{bnma}\\
\item the Brill-Noether map
$$
\mu_{o,C}:H^0(C,\o_C(1))\otimes H^0(C,\omega_C(-1))\to H^0(C,\omega_C)
$$
is surjective. \label{bnmb}
\end{enumerate}

Then $\Sigma$ has the expected number of moduli equal to
$3g-3+\rho-k$.
\end{proposition}

\begin{proof}
The case $k=0$ has been proved by Sernesi in \cite{ser}, section 4.
We shall assume $k>0$.
Let $\g$ be a plane curve verifying the hypotheses of the proposition.
By lemma 1.5.(b) of \cite{ta}, the hypothesis that $\Sigma$ is smooth
of the expected dimension at $\ge$ implies the vanishing $H^1(C,\fn)=0$, where $\fn$ if 
the normal sheaf of $\phi$. We recall that, denoting by $\Theta_{C}$ and 
$\Theta_{\p}$ the tangent sheaf of $C$ and $\p$ respectively, then the 
normal sheaf of $\phi$ is defined as the cokernel of the differential map $\phi_*$ of $\phi$
\begin{equation}\label{ns}
 0 \rightarrow \Theta_{C} \stackrel{\phi_*}\rightarrow  \phi^*\Theta_{\p}
\rightarrow { \mathcal N}_\phi  \rightarrow  0
\end{equation}

\noindent 
By theorem 3.1 of \cite{hor1}, the
vanishing $H^1(C,\fn)=0$ is a sufficient condition for the
existence of a universal deformation family
$$\xymatrix{
 \mathcal C\ar[r]^{\tilde\phi}\ar[d]_\pi & \p \\
B&}$$
of the normalization map $\phi$, whose parameter space $B$ is smooth
at the point $0$ corresponding to $\phi$, with tangent space at $0$ 
equal to $H^0(C,\,\fn)$. On the contrary, by \cite{ac3}, p. 487,
the Severi variety $\v=\Sigma^{n}_{0,k+d}$
of irreducible plane curves of genus $g=\pa-d-k$ is singular at the point
$\ge$ and the universal deformation space $B$ of $\phi$ is a desingularization
of $\v$ at $\ge$. Moreover, by corollary 6.11 of \cite{ac2},
if $B_k=F^{-1}(\Sigma)$ is the
locus of points of $B$ corresponding to a morphism with $k$
ramification points, then the tangent space to $B_k$
at $0$ is a subspace $W$ of $H^0(C,{ \mathcal
N}_\phi)$ of codimension $k$ such that $W\cap H^0(C, \k)=0$, where
$\k$ is the torsion subsheaf of $\fn$. By \cite{ac3}, p.487,
it follows that, if
$$
F:B\to\v
$$
is the natural $(1:1)$-map from $B$ to $\v$, 
then the differential map $$dF:H^0(C,\,\fn)\to T_{\ge}\v$$
restricts to an isomorphism  between $W$ and the
tangent space $T_{\ge}\Sigma$ to $\Sigma$ at $\ge$.

We can now go back to the number of moduli of $\Sigma$. From the
exact sequence \eqref{ns}, by using that $H^1(C,\,\fn)=0$, we get the 
following long exact sequence
$$
 0 \rightarrow H^0(C,\Theta_{C})\rightarrow H^0(C, \phi^*\Theta_{\p})
\rightarrow H^0(C,{ \mathcal N}_\phi)\stackrel {\delta_C} \rightarrow H^1(C, \Theta_{C})
\rightarrow H^1(C, \phi^*\Theta_{\p})\rightarrow  0
$$
Recalling that the space $H^1(C,
\Theta_{C})$ is canonically identified with the tangent space
$T_{[C]}\m_g$ to $\m_g $ at the point associated to the
normalization $C$ of  $\g$, the coboundary map $\delta_C:H^0(C,{
\mathcal N}_\phi)\rightarrow H^1(C, \Theta_{C})$ sends the Horikawa
class  of an infinitesimal deformation of $\phi$ to the 
Kodaira-Spencer class of the corresponding infinitesimal deformation of
$C$. So, $\d_C|W$ is the differential map at the point $0\in B$ of
the moduli map
$\Pi_{\Sigma}\circ F:B_k=F^{-1}(\Sigma)\dashrightarrow\m_g.$ Since the
point $[\g]$ is general in $\Sigma$, and recalling the isomorphism
$dF:W\stackrel{\backsim}\rightarrow T_{\ge}\Sigma$, we have that
$$
\textit{the number of moduli of $\Sigma$}=\dim (\d_C(W)).
$$
Now, from the exact sequence \eqref{ns}, we have that
$$
\dim (\d_C( H^0(C,{ \mathcal N}_\phi))=3g-3-h^1(C, \phi^*\Theta_{\p}).
$$
Moreover, from the pull-back to $C$ of the Euler exact sequence,
we deduce the well known isomorphism
$$H^1(C, \phi^*\Theta_{\p})\simeq\textrm{coker}(\mu_{0,C}^*)\simeq
(\ker(\mu_{0,C}))^*$$
and we conclude that
\begin{equation}\label{mrmm}
\dim (\d_C( H^0(C,{ \mathcal N}_\phi)))=3g-3-\dim (\ker(\mu_{0,C})).
\end{equation}
Notice that the previous equality is always true, even if $\g$
doesn't verify the hypothesis $(1)$ or $(2)$ of the statement. Moreover, if $\g$
is geometrically linearly normal, i.e. if $h^0(C,\o_C(1) )=3$, we
have that
$$
\rho=3n-2g-6=\dim (\textrm{coker} (\bnm))-\dim (\ker (\bnm)).
$$
When $\bnm$ is surjective, $\rho=-\dim ( \ker (\bnm))$ and
\begin{equation}\label{thesis}
\dim (\d_C( H^0(C,{ \mathcal N}_\phi))=3g-3+\rho=\dim (B)-8=\dim (\v)-8.
\end{equation}
Since the dimension of the fibre of the moduli map
$$\Pi_{\v}\circ F: B\dashrightarrow\m_g$$ has dimension at least equal to $8=dim(Aut(\p))$,
from \eqref{thesis} we deduce that the differential map of
$\Pi_{\v}\circ F$ has maximal rank at $0$ and, in
particular, we have that $\dim ((\Pi_{\v}\circ F)^{-1}([C]))=8$.
Equivalently, there exist only finitely many $g^2_n$ on $C$. It
follows that there are only finitely many $g^2_n$ on $C$ mapping
$C$ to the plane as a curve with $k$ cusps and $d$ nodes.
Then,
$$\dim(\delta_c(W))=\dim(\Pi_{\Sigma}(\Sigma)) =3g-3+\rho-k.$$

\end{proof}

\begin{remark}
Arguing as in the proof of the previous proposition, it has been
proved in \cite{ser} that, if\, $\g$ is a geometrically linearly normal
plane curve with only $d$ nodes as singularities and the
Brill-Noether map $\bnm$ of the normalization morphism of $\g$ is
injective, then $\Sigma=\Sigma_{0,d}^n$ has general moduli. If $\Sigma\subset\s$
and $\ge\in\Sigma$ verify the hypotheses of proposition \ref{bnm} but
we assume that $\bnm$ is injective, we may
only conclude that $\Pi_{\v}\circ F$ is dominant with
surjective differential map at $\ge$. So
$\dim(\Pi^{-1}_{V_{n,g}}([C]))=\rho+8$. But this is not useful to
compute the dimension of $\d_C(W)=\d_C(T_{\ge}\Sigma)$. However,
in this case we get that
$$
\d_C(T_{\ge}\Sigma)+\d_C(H^0(C,\k))=
\d_C(H^0(C,\fn))=H^1(C,\Theta_C).
$$
Then, by using that $\dim(\d_C(H^0(C,\k)))\leq k$ and by recalling
that if $\Sigma$ has the expected dimension then the number of moduli of $\s$ is at most the
expected one (see lemma \ref{lemmaubnm} and remark \ref{remarkubnm}), we find that
$$
3g-3-k\leq \textit{number of moduli of } \Sigma\leq 3g-3+\rho-k.
$$
\end{remark}

\begin{remark}
Notice that, if a plane curve $\g$ of genus $g$ verifies the
hypotheses \eqref{bnma} and \eqref{bnmb} of the previous
proposition, then the Brill-Noether number $\rho(2,g,n)$ is not
positive and, in particular, $g\geq 3$. We don't know examples of
complete irreducible families $\Sigma\subset \s$ with the expected
number of moduli whose general element $[\g]$ corresponds to a
curve $\g$ of genus $g$, with $\rho(2,g,n)\leq 0$, which doesn't
verify properties \eqref{bnma} and \eqref{bnmb}. 
\end{remark}

\begin{lemma}[\cite{ag}, Corollary 3.4]\label{ag}
Let $\g$ be an irreducible plane curve of degree $n$ with only
nodes and cusps as singularities and let $\phi:C\to\g$ be the
normalization morphism of $\g$. Suppose that $\g$ is geometrically
$2$-normal, i.e. $h^0(C,\o_C(2))=6$. Then the Brill-Noether map
$$
\mu_{o,C}:H^0(C,\mathcal O_C(1))\otimes H^0(C,\omega_C(-1))\to
H^0(C,\omega_C)
$$
is surjective.
\end{lemma}

\begin{proof}
By lemma \ref{gknpc}, the curve $\g$ is geometrically $2$-normal
if and only if the scheme $N$ of the singular points of $\g$
imposes independent linear conditions to the linear system
$H^0(\p,\o_{\p}(n-5))$ of plane curves of degree $n-5$. 
Since $H^0(\p,\o_{\p}(n-5))\subset H^0(\p,\o_{\p}(n-4))$,
$N$ imposes independent linear conditions plane curves of degree $n-4$, and,
by using lemma \ref{gknpc}, we get that
$h^0(C,\o_C(1))=3$, i.e. $\g$ is geometrically linearly normal.
Now, denote by $\mathcal I_{N|\p}$ the ideal sheaf of $N$. Notice 
that the curve $\g$ is geometrically
2-normal if and only if the ideal sheaf $\mathcal I_{N|\p}(n-4)$
is $0$-regular, (in the sense of Castelnuovo-Mumford). Indeed, since $h^2(\p,\mathcal I_{N|\p}(n-6))=0$,
the ideal sheaf $\mathcal I_{N|\p}(n-4)$ is $0$-regular if and
only if $h^1(\p,\mathcal I_{N|\p}(n-5))=0$. Because of the $0$-regularity of
$\mathcal I_{N|\p}(n-4)$, we have the surjectivity
of the natural map
$$
H^0(\mathbb P^2,\mathcal I_{N|\p}(n-4))\otimes H^0(\mathbb P^2,
\mathcal{O}_{\mathbb P^2}(1)) \rightarrow H^0(\mathbb P^2,
\mathcal I_{N|\p}(n-3)),
$$
(see \cite{mumford}).
Finally, by the geometric linear normality 
of $\g$, the vertical maps of the following commutative diagram
$$\xymatrix{
H^0(\mathbb P^2,\mathcal{O}_{\mathbb P^2}(1))\otimes
H^0(\mathbb P^2,\mathcal I_{N|\p}(n-4))\ar[r]\ar[d] &
H^0(\mathbb P^2,\mathcal I_{N|\p}(n-3))\ar[d] \\
H^0(C,\mathcal O_C(1))\otimes H^0(C,\omega_C(-1)) \ar[r]^-{\mu_{o,C}}
& H^0(C, \omega_C)}$$
are surjective and, hence, the Brill-Noether
map $\bnm$ is surjective too.
\end{proof}

\begin{corollary}\label{g2n}
Let $\Sigma\subset\s$ be an irreducible
component of $\s$ of dimension equal to $3n+g-1-k$, 
such that the general point $[\g]\in\Sigma$
corresponds to a geometrically $2$-normal plane curve. Then
$\Sigma$ has the expected number of moduli equal to $3g-3+\rho-k$.
\end{corollary}
\begin{proof}
It follows from proposition \ref{bnm} and lemma \ref{ag}.
\end{proof}
In order to produce examples of families of irreducible plane curves with
nodes and cusps with the expected number of moduli, we study how
increases the rank of the Brill-Noether map by smoothing a node or
a cusp of the general curve of the family, (in the sense of section \ref{S-E}).

Let $\Sigma\subset\s$, with $n\geq 5$, be an irreducible component
of $\s$, let $\ge\in\Sigma$ be a general point of $\Sigma$ and let
$\phi:C\to\g$ be the normalization of $\g$. 
Choose a singular point
$P\in\g$ and denote by $\phi^\prime:C^\prime\to\g$ the partial
normalization of $\g$ obtained by smoothing all singular points of
$\g$, except the point $P$. If $\omega_{C^\prime}$ is the
dualizing sheaf of $C^\prime$ and

$$
\mu_{o,C^\prime}:H^0(\cp,\o_{\cp}(1))\otimes H^0(\cp,\omega_{\cp}(-1))
\to H^0(\cp,\omega_{\cp}),
$$
is the natural multiplication map, we have the following result.

\begin{lemma}\label{rank}
If $h^0(C,\o_C(1))=3$ and the geometric genus $g$ of $C$ is such
that $g>n-2$, with $n\geq 5$, then $rk( \mu_{o,C^\prime})\geq
rk(\bnm)+1$. In particular, if $h^0(C,\o_C(1))=3$, $n\geq 5$ and
$\bnm$ is surjective, then $\mu_{o,C^\prime}$ is also surjective.
\end{lemma}

\begin{proof} Let $\psi:C\to\cp$ be
the normalization map.

$$\xymatrix{
C\ar[r]^{\psi}\ar[dr]_\phi& \cp\ar[d]^{\phi^\prime} \\
&\g}$$

\noindent We recall that, if we set $\phi^*(P):=p_1+p_2$ when $P$
is a node and $\phi^*(P)=2\phi^{-1}(P)$ when $P$ is a cusp, then
the dualizing sheaf of $C^\prime$ is a subsheaf of
$\psi_*(\omega_C(\phi^*(P)))$,  (see
for example \cite{eh2}, p.80).
In particular we have the following exact sequence
\begin{equation}\label{dualizingsheaf}
0\to\omega_{\cp}\to\psi_*\omega_C(\phi^*(P))\to\mathbb C_P\to 0
\end{equation}
where $\mathbb C_P$ is the skyscraper sheaf on $C$ with support at
$P$. From this exact sequence, we deduce that
$$H^0(\cp,\omega_{\cp})\simeq H^0(C,\omega_C(\phi^*(P))).$$
Moreover, tensoring \eqref{dualizingsheaf} by $\o_{\cp}(-1)$, we
find the exact sequence

\begin{equation}
0\to\omega_{\cp}(-1)\to\psi_*\omega_C(\phi^*(P))(-1)\to\mathbb
C_P\to 0
\end{equation}
from which we get an injective map $H^0(\cp,\omega_{\cp}(-1))\to
H^0(C, \omega_C(\phi^*(P))(-1)).$ On the other hand
\begin{equation}\label{fattoprima}
h^0(\cp,\omega_{\cp}(-1))=h^0(C, \omega_C(\phi^*(P))(-1))=g-n+3
\end{equation}
and so $$H^0(\cp,\omega_{\cp}(-1))\simeq H^0(C,
\omega_C(\phi^*(P))(-1)).$$ Moreover, from the hypothesis
$h^0(C,\o_C(1))=3$, we have that $H^0(C,\o_C(1)) \simeq  H^0(\cp,
\o_{\cp}(1))\simeq H^0(\p,\o_{\p}(1))$. Therefore, in the
following commutative diagram

$$\xymatrix{
H^0(\cp,\o_{\cp}(1))\otimes H^0(\cp,\omega_{\cp}(-1))
\ar[r]^-{\mu_{o,C^\prime}}\ar[d]& H^0(\cp,\omega_{\cp})\ar[d]\\
H^0(C,\o_C(1))\otimes H^0(C,\omega_C(-1)(\phi^*(P)))
\ar[r]^-{\mu_{o,C}^\prime}& H^0(C,\omega_C(\phi^*(P)))}$$
\noindent
where we denoted by $\mu_{o,C}^\prime$ the natural multiplication
map, the vertical maps are isomorphisms. In particular,
$$rk(\mu_{o,C^\prime})=rk(\mu_{o,C}^\prime).$$ In order to compute
the rank of $\mu_{o,C}^\prime$, we consider the following
commutative diagram

$$\xymatrix{
H^0(C,\o_C(1))\otimes H^0(C,\omega_C(-1))\ar[r]^-{\mu_{o,C}}\ar[d]&
H^0(C,\omega_{C})\ar[d]^G\\
H^0(C,\o_C(1))\otimes H^0(C,\omega_C(-1)(\phi^*(P)))
\ar[r]^-{\mu_{o,C}^\prime} & H^0(C,\omega_C(\phi^*(P)))}$$
\noindent
where the vertical maps are injections. Notice that, since we
supposed $n\geq 5$, $h^0(C,\o_C(1))=3$ and $g>n-2\geq 3$, the
sheaf $\o_C(1)$ is special. We deduce that $C$ is not
hyperelliptic and, chosen a basis of $H^0(C, \omega_C)$, the
associated map $C\to\mathbb P^{g-1}$ is an embedding. On the
contrary, the sheaf $\omega_C(\phi^*(P))$ does not define an
embedding on $C$. Choosing a basis of $H^0(C, \omega_C(\phi^*(P))$
and denoting by $\Phi:C\to\mathbb P^g$ the associated map, this
will be an embedding outside $\phi^*(P)$. If $P$ is a node of $C$
and $\phi^*(P)=p_1+p_2$, the image of $C$ to $\mathbb P^g$, with
respect to $\Phi$, will have a node at the image point $Q$ of
$p_1$ and $p_2$. If $P\in\g$ is a cusp, then $\Phi(C)$ will have a
cusp at the image point $Q$ of $\phi^{-1}(P)$. The hyperplanes of
$\mathbb P^g$ passing through $Q$ cut out on $C$ the canonical
linear series $|\omega_C|.$ Moreover, if we denote by
$B\subset\mathbb P^g$ the subspace which is the base locus of the
hyperplanes of $\mathbb P^g$ corresponding to
$Im(\mu_{o,C}^\prime)$, then $Q\notin B$. Indeed, $B$ intersects
the curve $C$ in the image of the base locus of
$|\o_C(1)|+|\omega_C(\phi^*(P))(-1)|:=\mathbb
P(Im(\mu^{\prime}_{0,C}))$, which coincides with the base locus of
$|\omega_C(\phi^*(P))(-1)|$, since $|\o_C(1)|$ is base point free.
Now, by \eqref{fattoprima},

\begin{equation*}
h^0(\omega_C(\phi^*(P))(-1))= 3+g-n=h^0(C, \omega_C(-1))+1.
\end{equation*}
Then $\phi^*(P)$ does not belong to the base locus of
$|\omega_C(\phi^*(P))(-1)|$, and so $$\dim(<Q,B>_{\mathbb
P^g})=\dim(B)+1.$$ Finally, we find that

\begin{eqnarray*}
rk(\bnm)=rk(G\bnm)
& \leq  & \dim(Im(G)\cap Im(\mu^\prime_{o,C}))\\
 &\leq &g+1-\dim(<B,Q>_{\mathbb P^g})-1\\
 &= &  g-1-\dim(B)\\
 &= & rk(\mu^\prime_{o,C})-1.
\end{eqnarray*}
\end{proof}

\begin{corollary} \label{semicon}
Let $\Sigma\subset\s$ be a non-empty irreducible component of the expected 
dimension of $\s$, with $n\geq 5$. Suppose that $\Sigma$ has the 
expected number of moduli and that the general element $\ge\in\Sigma$
corresponds to a g.l.n. plane curve $\g$ of geometric genus $g$
such that, if $C\to\g$ is the normalization of $\g$, then the map
$\bnm$ is surjective. Then, for every $k^\prime\leq k$ and
$d^\prime\leq d+k-k^\prime$, there is at least an irreducible
component $\Sigma^\prime\subset\Sigma_{k^\prime,d^\prime}^n$, such
that $\Sigma\subset\Sigma^\prime$, the general element
$[D]\in\Sigma^\prime$ corresponds to a g.l.n. plane curve $D$ of
geometric genus $g^\prime$ with normalization $D^\nu\to D$ and the
Brill-Noether map $\mu_{0,D^\nu}$ surjective. In particular,
$\Sigma^\prime$ has the expected number of moduli.
\end{corollary}
\begin{proof}
Let $\g$ be the curve corresponding to the general element $\ge$
of $\Sigma\subset\s$. Since by hypothesis $\Sigma$ is smooth of the 
expected dimension at $\ge$, by section \ref{S-E},
for every $k^\prime\leq k$ and for every $d^\prime\leq d+k-k^\prime$ there 
exists an irreducible component $\Sigma^\prime$ of
$\Sigma_{k^\prime,d^\prime}^n$ containing $\Sigma$. 
In order to prove the statement, it is enough to show it under the hypotheses $k^\prime=k-1$
and $d^\prime=d+1$, $k=k^\prime$ and $d^\prime=d-1$ or $d=d^\prime$ and $k^\prime=k-1$.
If $k^\prime=k-1$ and $d^\prime=d+1$, then the statement follows 
by standard semiconinuity arguments. If $k=k^\prime$ and $d^\prime=d-1$ or $d=d^\prime$ and $k^\prime=k-1$, the statement follows by lemma \ref{rank} and by standard semicontinuity
arguments.
\end{proof}

The following lemma has been stated and proved by Sernesi in
\cite{ser}. Actually, Sernesi supposes that $\g$ has only nodes as
singularities. But, since his proof works for plane curves $\g$
with any type of singularities and, since we need it for curves
with nodes and cusps, we state the lemma in a more general form.

\begin{lemma}\label{sernesi}(\cite{ser}, lemma 2.3)
Let $\g$ be an irreducible and reduced plane curve of degree
$n\geq 5$ with any type of singularities. Denote by $C$ the
normalization of $\g$. Suppose that $h^0(C, \o_{C}(1))=3$ and the
Brill-Noether map
$$
\bnm:H^0(C,\o_{C}(1))\otimes H^0(C,\omega_{C}(-1))\to H^0(C,\omega_{C}),
$$
has maximal rank. Let $R$ be a general line and let $P_1$, $P_2$
and $P_3$ be three fixed points of $\g\cap R$. We denote by $\cp$
the partial normalization of $\g^\prime=\g\cup R$, obtained by
smoothing all the singular points, except $P_1$, $P_2$ and $P_3$.
Then $h^0(\cp, \o_{\cp}(1))=3$ and, denoting by $\omega_{\cp}$ the
dualizing sheaf of $\cp$, the multiplication map
$$
\mu_{o,C^\prime}:H^0(\cp,\o_{\cp}(1))\otimes H^0(\cp,
\omega_{\cp}(-1))\to H^0(\cp,\omega_{\cp}),
$$
has maximal rank.
\end{lemma}

\begin{theorem}\label{moduli}
Let $\s$ be the algebraic system of irreducible plane curves of degree $n\geq 4$
with $k$ cusps, $d$ nodes and geometric genus $g=\pa-k-d$. Suppose that:
\begin{equation}\label{genus}
n-2\leq g\,\,\, \textrm{equivalently}\,\,\, k+d\leq h^0(\p,
\o_{\p}(n-4))
\end{equation}
and
\begin{equation}\label{g3n}
k\leq 6+\left[\frac{n-8}{3}\right]\,\,\textrm{if}\,\, 3n-9\leq
g\,\, \textrm{and}\,\, n\geq 6,
\end{equation}
\begin{equation}\label{otherwise}
k\leq 6\,\,\, \textrm{ otherwise}.
\end{equation}
Then $\s$ has at least an irreducible component $\Sigma$ which is
not empty and such that, if $\g\subset\p$ is the curve corresponding  
to the general element of $\Sigma$ and $C$ is the normalization curve
of $\g$, then $h^0(C,\o_C(1))=3$ and the map $\bnm$
 has maximal rank. In particular, when $\rho\leq
0$, the algebraic system $\Sigma$ has the expected number of
moduli equal to $3g-3+\rho-k$.
\end{theorem}
\begin{proof}
\textit{Suppose that \eqref{g3n} holds.} Then, by observing that
\begin{displaymath}
g\geq 3n-9\,\,\,\textrm{if and only if}\,\,\,k+d\leq
h^0(\p,\o_{\p}(n-6))
\end{displaymath}
and by using theorem \ref{gknnc} for $t=3$, we have that there
exists an irreducible component $\Sigma$ of $\s$ whose general
element is a geometrically $3$-normal plane curve $\g$.
By remark \ref{kimplies}, it follows that also the
linear systems cut out on $C$ by the conics and the lines are
complete. The statement follows from corollary \ref{g2n}.

 \textit{In order to prove the theorem
under the hypothesis \eqref{otherwise}, we consider the following
subcases:}
\begin{enumerate}
\item $2n-5\leq g\leq 3n-9$,\,\,\, \textrm{i.e.}\,\,\,
$h^0(\o_{\p}(n-6)) \leq k+d\leq h^0(\o_{\p}(n-5))$\,\,\,
\textrm{and}\,\,\, $n\geq 5$,\label{bg2n}\\
\item $n-2\leq g\leq 2n-7$\,\,\,\textrm{and}\,\,\, $n\geq 5$,
\label{gln}\\
\item $g= 2n-6$\,\,\,\textrm{and}\,\,\, $n\geq 4$.\label{casospeciale}
\end{enumerate}

\textit{Suppose that \eqref{bg2n} holds.} By theorem \ref{gknnc}
for $t=2$, we know that, under this hypothesis, there exists a
nonempty component $\Sigma\subset\s$, whose general element is
geometrically $2$-normal. We conclude as in the previous case, by
corollary \ref{g2n}.

\textit{Now, suppose that $g$ and $n$ verify \eqref{gln}. We shall
prove the theorem by induction on $n$ and $g$. Set $g=2n-7-a$,
with $a\geq 0$ fixed.} Suppose that the theorem is true for the
pair $(n,g)$, with $n\geq 7$. We shall prove the theorem for
$(n+1, g+2)$, observing that $g+2=2(n+1)-7-a$. 
Let $\g$ be
a g.l.n. irreducible plane curve of degree $n$ and genus
$g=2n-7-a$ with $k\leq 6$ cusps, $d$ nodes and no more
singularities. Let $C$ be the normalization of $\g$. Suppose that
the Brill-Noether map $\bnm$ has maximal rank. Let $R\subset\p$ be
a general line and let $P_1$, $P_2$ and $P_3$ be three fixed
points of $\g\cap R$. By section \ref{S-E}, since $k\leq 6<3n$,
one can smooth the singular points $P_1,\,P_2,\,P_3$ and 
preserve the other singularities of $\g\cup R\subset\p$,
obtaining a family of plane curves $\mathcal C\to\Delta$
whose general fibre is irreducible, has degree $n+1$ and genus $g+2$. 
We conclude by lemma \ref{sernesi} and by standard semicontinuity arguments.

\textit{Now we prove the first step of the induction for $n\geq 7$.
If $n=7$, we get $0\leq a\leq 2$. Let $a=0$, i.e. $g=2n-7-a=7$.}
Let $\g$ be a g.l.n. irreducible plane
curve of degree $n=7$, of genus $g=n=7$ with $k\leq 6$ cusps and
nodes as singularities, such that no seven singular points
of $\g$ lie on an irreducible conic. To prove that there exists
such a plane curve, notice that, by applying theorem 
\ref{gknnc} for $t=1$, we get that,
for any fixed $k\leq 6$, there exists a g.l.n.
irreducible sextic $D$ of genus four with $k$ cusps and
$d=6-k$ nodes. Let $R_1,\dots,R_6$ be the singular points of $D$.
Since the points $R_1,\dots,R_6$ of $D$ impose independent linear
conditions to the conics, however we choose five singular points
$R_{i_1},\dots,R_{i_5}$ of $D$, with
$I=(i_1,\dots,i_5)\subset(1,\dots,6)$, there exists only
one conic $C_I$, passing through these points. Let us set
$S=\bigcup_I C_I\cap D$ and let $R$ be a line intersecting $D$
transversally at six points out of $S$. By Bezout theorem, no
seven singular points of $\g^\prime=D\cup R$ belong to an
irreducible conic. Moreover, if $\tilde D$ is the normalization of
$D$, if $Q_1,\dots,Q_4$ are four fixed points of $D\cap R$ and
$D^\prime$ is the partial normalization of $\g^\prime$ obtained by
smoothing the singular points except $Q_1,\dots,Q_4$, then, by the
following exact sequence
\begin{equation}\label{firstpage}
0\to\mathcal O_{R}(1)(-Q_1-\cdots-Q_4)\to\mathcal
O_{D^\prime}(1)\to \mathcal O_{\tilde D}(1)\to 0
\end{equation}
we find that $h^0(D^\prime,O_{D^\prime}(1))=3$. By 
section 2.1, one can smooth the singularities
$Q_1,\dots,Q_4$ and preserve the other singularities of $D\cup R$, 
getting a family of irreducible septics
$\mathcal G\to\Delta$ whose general fibre $\g$
is a geometrically linearly normal irreducible septic 
with $k$ cusps and $8-k$ nodes such that no seven singular point of 
$\g$ belong to an irreducible conic. 
Let, now, $C$
be the normalization of $\g$ and 
let $\Delta\subset C$ be the adjoint divisor
of the normalization map $\phi:C\to\g$.
We shall prove that $\ker(\bnm)=0$. Since $\g$ is geometrically linearly normal,
we have that
$$
h^0(C,\omega_C(-1))=h^0(C,\o_C(3)(-\Delta)))=g-n+2=2.
$$
Then, by the \pencil, we find that
$$\ker(\bnm)=H^0(C,\omega_{C}^*(B)\otimes\o_C(2)),$$
where $B$ is the base locus of $|\omega_C(-1)=\o_C(3)(-\Delta)|$.
Let $\mathcal F$ be the pencil of plane cubics passing through the
eight double points $P_1,\dots,P_8$ of $\g$
and let $B_\mathcal F$ be the base locus of the pencil $\mathcal
F$. Let $\g_3$ be the general element of $\mathcal F$. Suppose
that $B_{\mathcal F}$ has dimension one. If $B_{\mathcal F}$
contains a line $l$, then, by Bezout theorem, at most three points
among $P_1,\dots,P_8$, say $P_1,\dots,P_3$ can lie on $l$ and the
other points have to be contained in the base locus of a pencil of
conics $\mathcal F^\prime$. Using again Bezout theorem, we find
that the curves of $\mathcal F^\prime$ are reducible and the
base locus of $\mathcal F^\prime$ contains a line $l^\prime$. But
also $l^\prime$ contains at most three points of $P_4,\dots,P_6$.
It follows that there is only one cubic through $P_1,\dots,P_8$.
This is not possible by construction. Suppose that $B_{\mathcal
F}$ contains an irreducible conic $\g_2$. By Bezout theorem, at
most seven points among $P_1,\dots,P_8$ may lie on $\g_2$. On the
other hand, since $\dim(\mathcal F)=1$, there are exactly seven
points of $P_1,\dots,P_8$, say $P_1,\dots,P_7$, on $\g_2$ and the
general cubic $\g_3$ of $\mathcal F$ is union of $\g_2$ and a line
passing through $P_8$. Since, by construction, no seven singular
points of $\g$ lie on an irreducible conic, also in this case we get a
contradiction. So the general element $\g_3$ of $\mathcal F$ is
irreducible. Using again Bezout theorem, we find that $\g_3$ is
smooth and $\mathcal F$ has only
one more base point $Q$. We consider the following cases:\\
a) $Q$ doesn't lie on $\g$;\\
b) $Q$ lies on $\g$, but $Q\neq P_1,\dots,P_8$;\\
c) $Q$ is infinitely near to one of the points $P_1,\dots,P_8$,
say $P_{\hat{i}}$, i.e. the cubics of $\mathcal F$ have at
$P_{\hat{i}}$ the same tangent line $l$, but $l$ is not contained
in
the tangent cone to $\g$ at $P_{\hat{i}}$;\\
d) $Q$ is like in the case c), but $l$ is contained in the tangent
cone to $\g$ at $P_{\hat{i}}$.\\
Suppose that the case a) or c) holds. Thus $B=0$ and
$$
\ker(\bnm)=H^0(C,\omega_{C}^*\otimes\o_C(2))=H^0(C,\o_{C}(-2)(\Delta)).
$$
By Riemann-Roch theorem,
$h^0(C,\o_{C}(-2)(\Delta))=h^0(C,\o_{C}(6)(-2\Delta))-4$. One sees that
$h^0(C,\o_{C}(6)(-2\Delta))=4$, by blowing-up the plane at
$P_1,\dots,P_8$ and by using some standard exact sequences. 
Suppose now that the case b) holds. Thus $B=Q$
and
$$\dim(\ker(\bnm))=h^0(C,\o_{C}(-2)(\Delta+Q))=h^0(C,\o_{C}(6)(-2\Delta-Q))-3.$$
Also in this case one sees that $h^0(C,\o_{C}(6)(-2\Delta-Q))=3$ by blowing-up  
at $P_1,\dots,P_8$ and $Q$ and by using standard exact sequences. 
Finally, we analyze the case d). Let $\Phi: S\to \p$ be the blow-up
of the plane at $P_1,\dots,P_8$ with exceptional divisors
$E_1,\dots,E_8$. Let $Q\in E_{\hat{i}}$ be the intersection point
of $E_{\hat{i}}$ and the strict transform $C_3$ of the general
cubic $\g_3$ of the pencil $\mathcal F$. We denote by $\tilde \Phi
:\tilde S\to S$ the blow-up of $S$ at $Q$ and by $\Psi:\tilde S\to
\p$ the composition map of the maps $\Phi$ and $\tilde \Phi$. We
still denote by $E_1,\dots,E_8$ their strict transforms on $\tilde
S$, by $C$ and $C_3$ the strict transforms of $\g$ and $\g_3$ and
by $E_Q$ the new exceptional divisor of $\tilde S$. In this case
we have that $\Psi^*(\g)=C+2\sum_iE_i+3E_Q$,
$\Psi^*(\g_3)=C_3+\sum_iE_i+2E_Q$. Moreover, the divisor
$\Delta$ is cut out on $C$ by $\sum_iE_i+E_Q$ and the base locus
$B$ of the linear series $|\omega_C(-1)|$ coincides with the
intersection point of $E_Q$ and $C$. So, we have that
$$
\dim(\ker(\bnm))=h^0(C,\o_C(-2)(\sum_iE_i+2E_Q))=h^0(C,\o_C(6)(-2\sum_iE_i-3E_Q))-3.
$$
Moreover, from the following exact sequence
$$
0\to\o_{\tilde S}(-1)\to\o_{\tilde
S}(6)(-2\sum_iE_i-3E_Q)\to\o_{C}(6)(-2\sum_iE_i-3E_Q)\to 0
$$
we find that
$H^0(C,\o_{C}(6)(-2\sum_iE_i-3E_Q))=H^0(\tilde S,\o_{\tilde
S}(6)(-2\sum_iE_i-3E_Q))$. In order to show that 
$h^0(\tilde S,\o_{\tilde
S}(6)(-2\sum_iE_i-3E_Q))=3$, we consider the 
following exact sequence
\begin{eqnarray}
0\to\o_{\tilde S}(3)(-\sum_iE_i-E_Q)&\to &\o_{\tilde
S}(6)(-2\sum_iE_i-3E_Q)\to \label{dernierm}\\
&\to &\o_{C_3}(6)(-2\sum_iE_i-3E_Q)\to 0\nonumber
\end{eqnarray}
By Riemann-Roch theorem, we have that
\begin{displaymath}
h^0(C_3,\o_{C_3}(6)(-2\sum_iE_i-3E_Q))=1\,\,\,\textrm{and}\,\,\,
h^1(C_3,\o_{C_3}(6)(-2\sum_iE_i-3E_Q)=0.
\end{displaymath}
Moreover, by Serre duality we have that
$$
H^1(\tilde S, \o_{\tilde S}(3)(-\sum_iE_i-E_Q)))= H^1(\tilde S,
\o_{\tilde S}(-6)(2\sum_iE_i+3E_Q))).
$$
From the exact sequence
\begin{equation}\label{lastequation}
0\to\o_{\tilde S }(-6)(+2\sum_iE_i+3E_Q))\to\o_{\tilde S
}(1)\to\o_C(1)\to 0
\end{equation}
by using that the map 
$H^0(\tilde S,\o_{\tilde S}(1))\to H^0(C,\o_C(1))$ 
is surjective and that \\
$h^1(\tilde S,\o_{\tilde S}(1))=0$, we find that $$H^1(\tilde S,\o_{\tilde S
}(-6)(+2\sum_iE_i+3E_Q)))=H^1(\tilde S,\o_{\tilde S
}(3)(-\sum_iE_i-E_Q)))=0.$$ Then, by
\eqref{dernierm}, 
$
h^0(\tilde S,\o_{\tilde S}(6)(-2\sum_iE_i-3E_Q))=
h^0(\tilde S,\o_{\tilde S}(3)(-\sum_iE_i-E_Q))+
h^0(C_3,\o_{C_3}(6)(-2\sum_iE_i-3E_Q))=3$
and $\ker(\bnm)=0$.
\textit{The first step of induction for $g=n=7$ and $k\leq 6$ is
proved.}

\textit{We complete the proof of the first step of the induction,
for $n$ and $g$ verifying \eqref{gln}.}
When $n=7$ and $1\leq a\leq 2$, the existence of a g.l.n. plane curve $\g$
follows from theorem \ref{gknnc}. Using the above
notation, $h^0(C,\omega_{C}(-1))=1$ if $a=1$ and
$h^0(C,\omega_{C}(-1))=0$ if $a=2$. In any case $\bnm$ is
injective. When $n\geq 8$ and $a\leq n-6$ the
theorem follows by induction from the case $n=7$. For $n\geq 8$
and $a=n-5$, we find that $g=n-2$, or, equivalently,
$k+d=h^0(\p,\o_{\p}(n-4))$. In theorem \ref{gknnc}, we
proved the existence of geometrically linearly normal plane curves
of degree $n\geq 8$ and genus $g=n-2$, with nodes and $k\leq 6 $
cusps. For every such plane curve $\g$, using the above notation,
the Brill-Noether map $\bnm$ is injective since
$h^0(C,\omega_{C}(-1))=0$. The cases $n=5$ and $n=6$
are similar.

\textit{Suppose now that $n$ and $g$ verify \eqref{casospeciale}}.
First of all we prove the theorem for $(n,g)=(4,2)$, $(5,4)$, $(6,6)$.
For $n=4$ and $g=2$, we find $n=g+2$ and we argue as in
the case $n\geq 8$ and $g=n-2$. Similarly, for $(n,g)=(5,4)$. For
$n=6$ and $g=6$ in theorem \ref{gknnc} we proved the existence of
geometrically linearly normal plane curves $\g$ with $k\leq 4$ 
cusps and nodes as
singularities. For every such a plane curve $\g$, denoting by $C$ its
normalization, we get that $h^0(C,\omega_{C}(-1))=2$, i.e. the
linear system $\mathcal F$ of conics passing through the four
singular points $P_1,\dots,P_4$ of $\g$ is a pencil which cuts out
on $C$ the complete linear series $|\omega_C(-1)|$.
We have two possibilities: either
the general element of this pencil is irreducible or it consists
of a line containing exactly three singular points $P_1,\,P_2,\,P_3$
of $\g$ and a line passing through $P_4$.
In any case the base locus of $\mathcal F$ intersects $\g$ only at  
$P_1,\dots,P_4$ and the linear series $|\omega_C(-1)|$ has no base
points. Then, by the \pencil\,, we find that
$\ker(\bnm)=H^0(C,\omega_C^*\otimes
\o(2))=H^0(C,\o_C(-1)(\Delta))$, where $\Delta\subset C$ is the
adjoint divisor of the normalization map $C\to\g$. By
Riemann-Roch theorem, we have that
$h^0(C,\o_C(-1)(\Delta))=h^0(C,\o_C(4)(-2\Delta))-3.$ 
By blowing-up at $P_1,\dots,P_4$, one can see that 
$h^0(C,\o_C(4)(-2\Delta))=3$, as we wanted.

Finally, we show the theorem under the hypothesis \eqref{casospeciale} 
for $n\geq 7$, by using induction on $n$. In order to prove the inductive step 
we may use lemma \ref{sernesi}, exactly as we did in the case \eqref{gln}.
We prove the first step of induction. If $n=7$ we have that $g=8$.
On pages \pageref{gln} and \pageref{firstpage} we proved the existence of geometrically
linearly normal plane curves $\g$ of degree $7$ and genus $7$ with
$k\leq 6$, such that, if $P_1,\dots,P_8$ are the singular points
of $\g$, then no seven points among $P_1,\dots,P_8$ lie on a
conic. In particular, we proved that, for every such a plane curve $\g$, 
the general element of the pencil of cubics passing through 
$P_1,\dots,P_8$ is irreducible and,
if $\phi:C\to\g$ is the normalization of
$\g$, then the Brill-Noether map $\bnm$ is injective. Let
$C^\prime$ be the partial normalization of $\g$ which we get by
smoothing all the singular points of $\g$ except a node, say
$P_8$. By using the
same notation and by arguing exactly as in the proof of lemma
\ref{rank}, we get the following commutative diagram

$$\xymatrix{
H^0(\cp,\o_{\cp}(1))\otimes H^0(\cp,\omega_{\cp}(-1))
\ar[r]^-{\mu_{o,C^\prime}}\ar[d]&
H^0(\cp,\omega_{\cp})\ar[d]\\
H^0(C,\o_C(1))\otimes H^0(C,\omega_C(-1)(\phi^*(P_8)))
\ar[r]^-{\mu_{o,C}^\prime}& H^0(C,\omega_C(\phi^*(P_8)))
}$$
where ${\mu}^\prime_{o,C}$ is the multiplication map and the
vertical maps are isomorphisms. We want to prove that the map
$\mu_{o,C^\prime}$ is surjective.  By the previous diagram it is
enough to prove that $\bnm^\prime$ is surjective. Since 
$h^0(C,\omega_C(\phi^*(P_8)))=8$ and 
$$h^0(C,\o_C(1))h^0(C,\omega_C(-1)(\phi^*(P_8)))=3(7-7+3)=9,$$ we have that 
$\dim(\ker(\mu_{o,C^\prime}))\geq 1$ and $\mu_{o,C^\prime}$ is surjective
if $\dim(\ker(\mu_{o,C^\prime}))=1$. By recalling that $\g$ is geometrically
linearly normal, we have that, if $Z$ is the scheme of the points 
$P_1,\dots,P_7$ and $\mathcal I_{Z|\p}$ is the ideal sheaf of $Z$ in $\p$, 
then in the following commutative diagram 
$$\xymatrix{
H^0(C,\o_C(1))\otimes H^0(C,\omega_C(-1)(\phi^*(P_8)))
\ar[r]^-{\mu_{o,C}^\prime}\ar[d]& H^0(C,\omega_C(\phi^*(P_8)))\ar[d]\\
H^0(\p,\o_{\p}(1))\otimes H^0(\p,\mathcal I_{Z|\p}(3))
\ar[r]^-{\mu}&
H^0(\p,\mathcal I_{Z|\p}(4))   
}$$
the vertical maps are isomorphisms. Hence, it is enough to prove that
the kernel of the multiplication map $\mu$ has dimension one. 
Let $\left\{f_0,f_1,f_2\right\}$
be a basis of the vector space $H^0(\p,\mathcal I_{Z|\p}(3))$. Since the general 
cubic passing through $P_1,\dots,P_8$ is irreducible, we may assume that $f_0,\,f_1$
and $f_2$ are irreducible. Suppose, by contradiction, that there exist at least 
two linearly independent vectors in the kernel of $\mu$. Then, there exist
sections $u_0,u_1,u_2$ and $v_0,v_1,v_2$ of $H^0(\p,\o_{\p}(1))$ such that
the sections $\sum_i u_i\otimes f_i$ and $\sum_i v_i\otimes f_i$ are linearly 
independent in $H^0(\p,\o_{\p}(1))\otimes H^0(\p,\mathcal I_{Z|\p}(3))$ and
\begin{equation}\label{sistfin}
\left\{
\begin{array}{ccc}
\sum_{i=0}^3 u_if_i&=&0\\
\sum_{i=0}^3 v_if_i&=&0.
\end{array}
\right. 
\end{equation}
We can look at \eqref{sistfin} as a linear system in the variables 
$f_0,\,f_1,\,f_2$. The space of solutions of \eqref{sistfin}
is generated by the vector $$(u_1v_2-u_2v_1,u_3v_0-u_0v_3,u_0v_1-u_1v_0).$$
In particular, if we set $q_i=(-1)^{1+i}u_iv_j-v_i u_j$, we find that
$f_j q_i=f_iq_j$, for every $i\neq j$. But this is not possible 
since $f_1,f_2$ and $f_3$ are irreducible. We deduce that $$\dim(\ker(\mu))=
\dim(\ker(\mu_{o,C^\prime}))=1$$ and $\mu_{o,C^\prime}$ is surjective. 
The existence of a plane septic of genus $8$ with $k\leq 6$ cusps
and nodes as singularities, with injective Brill-Neother map,
follows now by smoothing the node $P_8$ (in the sense of section \ref{S-E})
and by standard semicontinuity arguments.
\end{proof}

\begin{remark}\label{remarkmoduli}
Notice that the conditions which we found in theorem \ref{moduli}
in order that $\s$ has at least an irreducible component with the
expected number of moduli, are not sharp, even if we suppose
$\rho\leq 0$. To see this, notice that in remark
\ref{notsharpness} we proved the existence of an irreducible
component $\Sigma$ of $\Sigma^{12}_{9,0}$ whose general element
corresponds to a $3$-normal plane curve. By remark \ref{kimplies}
and corollary \ref{g2n}, we have that $\Sigma$ has the expected
number of moduli.
\end{remark}

\begin{theorem}\label{rho=1}
$\Sigma_{1,d}^n$ has the expected number of moduli, for every
$d\leq \pa-1$.
\end{theorem}
\begin{proof}
First of all, we recall that, by \cite{k2},
$\Sigma_{1,d}^n$ is irreducible for every
$d\leq \pa-1$. Moreover, from theorem \ref{moduli} and from
corollary \ref{5.431}, we know that $\Sigma_{1,d}^n$ is not empty and it has the
expected number of moduli if either $\rho\leq 0$ or $\rho\geq 2$. Next we
shall prove that, if $\rho=1$, then the algebraic system
$$\Sigma_{1,d}^n=\Sigma_{1,\frac{(n-3)^2}{2}-1}^n$$ has general
moduli. Equivalently, we will show that, if $\ge\in\Sigma_{1,d}^n$
is a general point and $g=\pa-1-d=\frac{3n-7}{2}$, then, on the
normalization curve $C$ of $\g$ there are only finitely many
linear series $g^2_n$ with at least a ramification point. Notice that, if $g=\pa-1-d=\frac{3n-7}{2}$, then $n$ is odd
and $n\geq 5$. We prove the statement by induction on $n$.
 
If $n=5$ then $g=4$. Let $C\subset\mathbb P^3$ be the canonical 
model of a general curve of genus four and let $2P+Q$, with $P\neq Q$
be a divisor in a $g^1_3$ on $C$. This divisor is cut out on $C$ by the 
tangent line to $C$ at $P$. The projection of $C$ from $Q$ is a plane quintic 
of genus four with a cusp. This proves 
that $\Sigma_{1,1}^5$ has general moduli. 

Now we suppose that the theorem is true for $n$ and we prove the
theorem for $n+2$. Let $\g\subset\p$ be the plane curve with a cusp
and $\frac{(n-3)^2}{2}-1$ nodes corresponding to a general point $\ge\in\Sigma^{n}_{1,\frac{(n-3)^2}{2}-1}$ and let $C_2$ be an irreducible
conic intersecting $\g$ transversally. By section \ref{S-E},
the point $[C_2\cup \g]$ belongs to $\Sigma^{n+2}_{1,\frac{(n+2-3)^2}{2}-1}$.
In particular, however we choose four points $P_1,\dots,\,P_4$
of intersection between $\g$ and $C_2$, there exists an analytic branch
$\mathcal S_{P_1,\dots,\,P_4}$ of $\Sigma^{n+2}_{1,\frac{(n-1)^2}{2}-1}$,
passing through $[C_2\cup \g]$ and whose general point corresponds to
an irreducible plane curve of degree $n+2$ with a cusp in a neighborhood
of the cusp of $\g$ and a node at a neighborhood of every node of $C_2\cup\g$
different from $P_1,\dots,\,P_4$. Moreover, $\mathcal S:=\mathcal S_{P_1,\dots,\,P_4}$ is
smooth at the point $[C_2\cup \g]$, (see \cite{tesi}, chapter 2). Let 
$$
\Pi:\Sigma^{n+2}_{1,\frac{(n-1)^2}{2}-1}\dashrightarrow\m_{\frac{3(n+2)-7}{2}}
$$
be the moduli map of $\Sigma^{n+2}_{1,\frac{(n-1)^2}{2}-1}$. In order to prove
that $\Pi$ is dominant it is sufficient to show that $\overline{\Pi(\mathcal S)}=\m_{\frac{3n-1}{2}}$.
By section \ref{S-E}, there exist an analytic open sets $\mathcal S^i\subset\Sigma^{n+2}_{1,\frac{(n-3)^2}{2}-1+2n-i}$,
with $i=1,\,2,\,3$, such that
$$
\mathcal S^0:=\mathcal S\cap (\mathbb P^5\times\Sigma^n_{1,\frac{(n-3)^2}{2}-1})
\subset\mathcal S^1\subset\mathcal S^2\subset \mathcal S^3\subset\mathcal S.
$$
Every $\mathcal S^i$, with $i=1,2,3$, has $\binom{4}{4-i}$ irreducible components,
passing through $[C_2\cup\g]$ and intersecting transversally at $[C_2\cup\g]$,
(see \cite{tesi}, chapter 2 or \cite{z2}). Moreover,
the general point of every irreducible component of $\mathcal S^i$, with $i=1,\,2,\,3$, corresponds to an irreducible plane curve $\g_i$ of degree $n+2$ with a cusp in a neighborhood of the cusp of $\g$, a 
node in a neighborhood of every node of $C_2\cup\g$ different from 
$P_1,\dots,\,P_4$ and $4-i$ nodes specializing to $4-i$ fixed points among $P_1,\dots,\,P_4$,
as $\g_i$ specializes to $C_2\cup\g$. Now, notice that the moduli map $\Pi$ is not defined
at the point $[C_2\cup\g]$, but, if $\mathcal S$ is sufficiently small, then the restriction of $\Pi$ to $\mathcal S$ extends to a regular function on $\mathcal S$.
More precisely, let $\mathcal C\to\Delta$ be any family of curves, parametrized by
a projective curve $\Delta\subset\mathcal S$, passing through the point $[C_2\cup\g]$ and 
whose general point corresponds to an irreducible plane curve of degree $n+2$
of genus $\frac{3n-1}{2}=\frac{3(n+2)-7}{2}$ with a cusp and nodes as singularities.
If we denote by $\mathcal C^\prime\to\Delta$ the family of curves obtained from $\mathcal C\to\Delta$ by normalizing the total space, we have that the general fibre of $\mathcal C^\prime\to\Delta$ is a smooth curve of genus $\frac{3n-1}{2}$, corresponding to
the normalization of the general fibre of $\mathcal C\to\Delta$, whereas the special fibre 
$\mathcal C^\prime_0$ is the partial normalization of $C_2\cup\g$, obtained by
normalizing all the singular points, except $P_1,\dots,\,P_4$. Then, the map $\Pi_{|_{\mathcal S}}$ is defined at $[C_2\cup\g]$ and it
associates to the point $[C_2\cup\g]$ the isomorphism class of $\mathcal C^\prime_0$.
Similarly, if $[\g_i]$ is a general point in one of the irreducible components of
$\mathcal S^i$, with $i=1,\,2,\,3$, then $\Pi_{|_{\mathcal S}}([\g_i])$ is the partial normalization of $\g_i$ obtained by smoothing all the singular points except for the $4-i$ nodes of $\g_i$ tending to $4-i$
fixed points among $P_1,\dots,P_4$ as $\g_i$ specializes to $C_2\cup\g$. 
It follows that, if we denote by $\m_{\frac{3n-1}{2}}^j$ the locus of 
$\m_{\frac{3n-1}{2}}$ parametrizing $j$-nodal curves, then $\Pi_{\mathcal S}(\mathcal S^{i})
\subseteq\m_{\frac{3n-1}{2}}^{4-i}$, for every $i=0,\dots,\,4$, and $\Pi_{\mathcal S}(\mathcal S^{i})
\varsubsetneq \Pi_{\mathcal S}(\mathcal S^{i+1})$. In particular, we find that
$$
\dim(\Pi_{|_{\mathcal S}}(\mathcal S))\geq \dim(\Pi_{|_{\mathcal S}}(\mathcal S^0))+4.
$$
In order to compute the dimension of $\Pi_{|_{\mathcal S}}(\mathcal S^0)$ we consider the 
rational map 
\begin{eqnarray*}
F:\Pi_{|_{\mathcal S}}(\mathcal S^0) & \dashrightarrow & \m_{\frac{3n-7}{2}}
\end{eqnarray*}
forgetting the rational tail. By the hypothesis that $\Sigma_{1,\frac{(n-3)^2}{2}-1}^n$
has general moduli and hence $F$ is dominant. Moreover, if $C$ is the normalization 
curve of $\g$, by the generality of $\ge$ in $\Sigma_{1,\frac{(n-3)^2}{2}-1}^n$, we may 
assume that $C$ is general in $\m_{\frac{3n-7}{2}}$. We want to show that $\dim(F^{-1}([C]))=5$. In order to see this, we recall that, by the hypothesis that $\Sigma_{1,\frac{(n-3)^2}{2}-1}^n$ has general moduli, on $C$ there exist only 
finitely many linear series of degree $n$ and dimension two, mapping $C$ to the plane as curve
with a cusp and nodes as singularities. Let $g^2_n$ be one of these linear series, let $\left\{s_0,\,s_1,\,s_2\right\}$ be a basis of $g^2_n$ and 
$\phi^\prime:C\to\g^\prime\subset\p$ the associated morphism.
If $Q_1,\dots,Q_4$ are four general points of $\g^\prime$,
then the linear system of conics through $Q_1,\dots,Q_4$ is a pencil $\mathcal F(Q_1,\dots,Q_4)$.
Let $C_2$ and $D_2$ be two general conics of $\mathcal F(Q_1,\dots,\,Q_4)$.
We claim that, if $\eta:\mathbb P^1\to C_2$ and $\beta:\mathbb P^1\to D_2$ are
isomorphisms between $\mathbb P^1$ and $C_2$ and $D_2$ respectively,
then the points $\eta^{-1}(Q_1),\dots,\eta^{-1}(Q_4)$ are not projectively equivalent 
to the points $\beta^{-1}(Q_1),\dots,\beta^{-1}(Q_4)$. In order to prove this, 
it is enough to prove that there are at least two conics in the 
pencil $\mathcal F(Q_1,\dots,Q_4)$ which verify the claim.
Let $D\subset\p$ be a conic. If we choose two sets of points $p_1,\dots,p_4$
and $q_1,\dots,q_4$ of $D$ not projectively equivalent on $D$, we may always find
projective automorphisms $A:\p\to\p$ and $A^\prime:\p\to\p$ such that $A(p_i)=Q_i$ and $A^\prime(q_i)=(Q_i)$, for every $i$. By construction, the conics $C_2=A(D)$ and $D_2=A^\prime(D)$
belong to the pencil $F(Q_1,\dots,Q_4)$ and verify the claim. 
This implies that
the partial normalizations $C^\prime$ and $D^\prime$ of $\g^\prime\cup C_2$ and $\g^\prime\cup D_2$,
obtained by smoothing all the singular points except $Q_1,\dots,Q_4$, are not isomorphic. 
Now, let $C_2^\prime$ be a general conic of $\mathcal F(Q_1,\dots,Q_4)$ and let $R_1,\dots,\,R_4$ be four general points  of $\g^\prime$, different from $Q_1,\dots,Q_4$.
If $D_2^\prime$ is a general conic of the pencil $\mathcal F(R_1,\dots,\,R_4)$, then 
the partial normalization $C^\prime$ and $D^\prime$ of $\g^\prime\cup C_2^\prime$
and $\g^\prime\cup D_2^\prime$ obtained, respectively, by smoothing all the singular points except $Q_1,\dots,Q_4$ and $R_1,\dots,R_4$, are not isomorphic. Indeed, since $C$
is a general curve of genus $\frac{3n-7}{2}\geq 7$, the only automorphism of $C$
is the identity. 
This proves that $\dim(F^{-1}([C]))=5$. In particular, we deduce that
$$\dim(\Pi_{|_{\mathcal S}}(\mathcal S^0))= 3\frac{3n-7}{2}-3+5
$$
and
$$
\dim(\Pi_{|_{\mathcal S}}(\mathcal S)\geq 3\frac{3n-7}{2}-3+9=3\frac{3(n+2)-7}{2}-3.
$$
\end{proof}
\begin{remark}
We expect that it is possible to prove that $\s$ has expected number of moduli 
for every $\rho$ also when $k=2$ or $k=3$. By corollary \ref{5.431}
and theorem \ref{moduli}, 
$\s$ is not empty, irreducible and it has expected
number of moduli for $\rho\leq 0$ and
$\rho\geq 2k$. In order to extend theorem \ref{rho=1}  to the case $k=2$ and
$k=3$ one needs to consider a finite number of cases. 
\end{remark} 

\subsection*{Acknowledgment}
The results of this paper are part of my PhD-thesis. I would like
to express my gratitude to my advisor Prof. C. Ciliberto who initiated
me into the subject of algebraic geometry and who provided me many
invaluable suggestions. I have also enjoyed and 
benefited from conversation with many people including F. Flamini,
E. Sernesi, L. Chiantini, L. Caporaso and G. Pareschi. Finally, I would like 
to thank the referee for useful remarks which allowed me to improve the finale version
of this paper. 

{}

\end{document}